\documentclass[12pt,a4paper]{article}

\usepackage{fullpage}
\usepackage[dvips]{graphicx}
\usepackage{psfrag}
\usepackage{amsmath}
\usepackage{fancybox}
\usepackage{amsfonts}
\usepackage{amsthm}
\usepackage{amssymb}
\usepackage{comment}

\usepackage[normalem]{ulem} 

\usepackage{enumerate}

 \DeclareGraphicsExtensions{.eps,.png,.pdf,.jpg}
 
 \theoremstyle{plain}
    \newtheorem{theorem}{Theorem}[section]
    \newtheorem{corollary}[theorem]{Corollary}
    \newtheorem{lemma}[theorem]{Lemma}
    \newtheorem{proposition}[theorem]{Proposition}
    
 \theoremstyle{definition}
    \newtheorem{definition}[theorem]{Definition}
    \newtheorem{remark}[theorem]{Remark}
 \theoremstyle{remark}

 \usepackage[english]{babel}

\usepackage[%
   bookmarks=true,
   pagebackref=false, 
   colorlinks,
   linkcolor=blue,
   anchorcolor=blue,
   citecolor=blue,
   filecolor=blue,
   urlcolor=blue
   ]{hyperref}

\usepackage{bm}
\usepackage{xspace}

\newcommand{\D}{\displaystyle}

\DeclareMathOperator{\supp}{supp}

\DeclareMathOperator{\Span}{span}

\DeclareMathOperator{\id}{id}
\DeclareMathOperator{\loc}{loc}

\newcommand\qih{\Pi}
\newcommand\NN{\mathbb N}
\newcommand\ZZ{\mathbb Z}

\newcommand\BB{\mathcal B}
\newcommand{\II}{\mathcal I}

\newcommand\QQ{\mathcal Q}
\newcommand\HH{\mathcal H}
\newcommand\VV{\mathcal S}

\newcommand\RRR{\mathcal D}

\newcommand\CC{\mathcal C}

\newcommand\PPP{\mathcal P}

\newcommand{\RR}{\mathbb R}

\usepackage{authblk}

\begin{document}

\title{New refinable spaces and local approximation estimates for hierarchical splines}

\author[1]{Annalisa Buffa}
\author[1,2,3]{Eduardo M. Garau \thanks{Corresponding author: egarau@santafe-conicet.gov.ar}}
\affil[1]{\footnotesize Istituto di Matematica Applicata e Tecnologie Informatiche 
`E. Magenes' (CNR), Italy}
\affil[2]{\footnotesize Instituto de Matem\'atica Aplicada del Litoral (CONICET-UNL), Argentina} 
\affil[3]{\footnotesize Facultad de Ingenier\'ia Qu\'imica (UNL), Argentina}

\maketitle

\begin{abstract}
 We study the local approximation properties in hierarchical spline spaces through multiscale quasi-interpolation operators. This construction suggests the analysis of a subspace of the classical hierarchical spline space~\cite{Vuong etal} which still satisfies the essential properties of the full space. The B-spline basis of such a subspace can be constructed using parent-children relations only, making it well adapted to local refinement algorithms. 
\end{abstract}

\begin{quote}\small
\end{quote}
%

\section{Introduction}

Local adaptivity in numerical methods for partial differential equations makes 
possible to solve real problems leading to a suitable approximation of the 
desired solution without exceeding the limits of available software. When 
considering isogemetric methods~\cite{HCB05,CHB09}, from a theoretical point of view, the design 
of efficient and robust strategies for local refinement constitutes a challenging 
problem because the tensor product structure of B-splines~\cite{deBoor,Schumaker} is broken.


Hierarchical B-splines (HB-splines) based on the construction presented in~\cite{Kraft-thesis,Kraft, Vuong etal} are a promising approach, because their
construction and properties are closely related to the ones of hierarchical
finite elements. Moreover, truncated hierarchical B-splines (THB-splines) have been introduced in~\cite{GJS12}, where their use as a framework for isogeometric analysis that provides local refinement possibilities has been analysed; see also~\cite{GJS14}. 
Local approximation estimates for hierarchical spline spaces have been studied in~\cite{SM14} using quasi-interpolants described in terms of the truncated hierarchical basis while the use of THB-splines in conjunction with residual based error indicators has been proposed in~\cite{BuGi15} under a few assumptions on the meshes. It is important to remark that truncation is indeed a possible strategy to recover
partition of unity and convex hull property.  On the other hand, the procedure of truncation requires a specific construction that entails complicated basis function supports (that may be non convex and/or not connected) and their use may produce a non negligible
overhead with an adaptive strategy.

In the present paper, we take the point of view of classical hierarchical B-splines and we study
their structure with a special attention to all those properties that may be needed or may facilitate
their use with an adaptive isogeometric method.

When considering an underlying sequence of nested tensor-product spline spaces and the corresponding B-spline bases, a particular way of selecting B-spline basis functions from each different level in order to build a hierarchical basis $\HH$ has been established in~\cite{Kraft-thesis} (see also~\cite{Vuong etal}). This hierarchical basis $\HH$ enjoys some important properties:
\begin{itemize}
\item $\HH$ is a set of linearly independent B-spline functions.
\item It is possible to identify uniquely the basis $\HH$ from the knowledge of a hierarchical mesh.
 \item All functions in the coarsest underlying tensor-product spline space belong to the hierarchical space $\Span\HH$.
 \item Under certain assumptions about the hierarchy of subdomains associated to $\HH$, it is possible to define a multiscale quasi-interpolant operator in $\Span\HH$ obtaining optimal orders of local approximation.
 \item Any enlargement of the hierarchy of the subdomains associated to $\HH$ gives rise to a \emph{refined} basis $\HH^*$ in the sense that $\Span\HH\subset\Span\HH^*$.
\end{itemize}

On the other hand, unlike tensor-product B-spline bases, the functions in the hierarchical basis $\HH$ do not constitute a partition of unity. If we consider the corresponding coefficients $\{a_\beta\}_{\beta\in\HH}\subset\RR$ in order to form such a partition, i.e.,
\begin{equation}\label{E:intro partition of unity}
 \sum_{\beta\in\HH} a_\beta \beta \equiv 1,
\end{equation}
it is known that the coefficients $a_\beta$ are nonnegative, but in fact, some of them can be equal to zero.  

In this article we analyse the local approximation properties of hierarchical splines spaces through the construction of a multiscale quasi-interpolant operator. Kraft~\cite{Kraft-thesis} has introduced such a operator for the case of bivariate spline spaces on  infinite uniform knot vectors and has studied its pointwise approximation properties. We extend his results to the case of open knot vectors with possible multiple internal knots in $d$-dimensional domains, for $d\ge 1$, and we also provide local approximation estimates in $L^q$-norms, for $1\le q\le\infty$.

Furthermore, we propose a  new hierarchical 
spline space, through a construction of a set of basis functions named $\tilde\HH$ that satisfies all important properties just listed above. In particular, $\tilde\HH\subset\HH$, and therefore, in general, the \emph{new} hierarchical space $\Span\tilde\HH$ may be smaller than $\Span\HH$. However, the new basis is easier to manage and to update when performing local adaptive refinement, and in this case, the coefficients for~\eqref{E:intro partition of unity} are strictly positive. 

This paper is organized as follows. In Section~\ref{S:tensor product} we introduce the notation and the assumptions for the underlying tensor-product spline spaces to be considered, and in Section~\ref{S:hierarchical basis} we briefly introduce the standard hierarchical B-spline basis and prove some results which will be useful later. In Section~\ref{S:interpolation} we construct a multiscale quasi-interpolant operator and study the local approximation properties in hierarchical spline spaces. In Section~\ref{S:new hierarchical basis} we define the new hierarchical B-spline basis and prove some of its basic properties.  Finally, we conclude the article with some final remarks in Section~\ref{S:remarks}.

\section{Spline spaces and B-spline bases}\label{S:tensor product}

\paragraph{Univariate B-spline bases}
Let $\Xi_{p,n}:=\{\xi_j\}_{j=1}^{n+p+1}$ be a $p$-open knot vector, i.e., a \emph{sequence} such that
\begin{equation*}\label{E:open knot vector}
 0=\xi_1=\dots=\xi_{p+1}<\xi_{p+2}\le\dots\le\xi_{n}<\xi_{n+1}=\dots=\xi_{n+p+1}=1,
\end{equation*}
 where the two positive integer $p$ and $n$ denote a given polynomial degree, and the corresponding number of B-splines defined over the subdivision $\Xi_{p,n}$, respectively. Here, $n\ge p+1$.
We also introduce the \emph{set} $Z_{p,n}:=\{\zeta_j\}_{j=1}^{\tilde n}$ of knots without repetitions, and denote by $m_j$ the multiplicity of the breakpoint $\zeta_j$, such that
$$\Xi_{p,n}=\{\underbrace{\zeta_1,\dots,\zeta_1,}_{m_1 \text{ times}}\underbrace{\zeta_2,\dots,\zeta_2,}_{m_2 \text{ times}}\dots \underbrace{\zeta_{\tilde n},\dots,\zeta_{\tilde n}}_{m_{\tilde n} \text{ times}}\},$$
with $\D\sum_{i=1}^{\tilde n} m_i = n+p+1$. Note that the two extreme knots are repeated $p+1$ times, i.e., $m_1=m_{\tilde n}=p+1$. We assume that an internal knot can be repeated at most $p+1$ times, i.e., $m_j\le p+1$, for $j=2,\dots,{\tilde n}-1$. 

Let $\BB(\Xi_{p,n}) := \{b_1,b_2,\dots,b_n\}$ be the B-spline basis (cf.~\cite{deBoor,Schumaker}) associated to the knot vector $\Xi_{p,n}$. The \emph{local knot vector} of $b_j$ is given by 
$$\Xi_{b_j}:= \{\xi_{j},\dots,\xi_{j+p+1}\},$$
which is a subsequence of $p+2$ consecutive knots of $\Xi_{p,n}$. We remark that 
$$\supp b_j = [\xi_{j},\xi_{j+p+1}].$$

Let $\II(\Xi_{p,n})$ be the mesh defined by
$$\II(\Xi_{p,n}):=\{[\zeta_j,\zeta_{j+1}]\,|\,j=1,\dots,{\tilde n}-1\}.$$
For each $I=[\zeta_j,\zeta_{j+1}]\in\II(\Xi_{p,n})$ there exists a unique  $k=\sum_{i=1}^{j} m_j$ such that $I=
[\xi_k,\xi_{k+1}]$ and $\xi_k\neq\xi_{k+1}$.  The union of the supports of the B-splines acting on $I$ identifies the \emph{support extension} $\tilde I$, namely
\begin{equation*}
 \tilde I := [\xi_{k-p},\xi_{k+p+1}].
\end{equation*}

Let $\Xi_0:=\Xi_{p,n}$ be a given $p$-open knot vector. We consider a sequence $\{\Xi_\ell\}_{\ell\in\NN}$ of sucessive refinements of $\Xi_0$, i.e.,
\begin{equation}\label{E:nested knot vectors}
 \Xi_0\subset\Xi_1\subset\dots,
\end{equation}
where $\Xi_\ell$ is a $p$-open knot vector, and  $\Xi_\ell\subset\Xi_{\ell+1}$ means that  $\Xi_\ell$ is a subsequence of $\Xi_{\ell+1}$, for $\ell\in\NN_0$. In other words, condition~\eqref{E:nested knot vectors} says that if $\xi$ is a knot in $\Xi_{\ell}$ with multiplicity $m$, then $\xi$ is also a knot in $\Xi_{{\ell+1}}$ with multiplicity at least $m$, for $\ell\in\NN_0$.

Let $\BB_\ell:=\BB(\Xi_\ell)$ be the B-spline basis, for $\ell\in\NN_0$.

\begin{definition}
 Let $\beta_\ell\in\BB_\ell$ be given and let $\Xi_{\beta_\ell}$ be the corresponding local knot vector. Let $\Xi_{\beta_\ell}^{(\ell+1)}\subset \Xi_{\ell+1}$ be the knot vector obtained from $\Xi_{\beta_\ell}$ after inserting the knots of  $\Xi_{\ell+1}$ which are in the interior of $\supp\beta_{\ell}$. We say that $\beta_{\ell+1}\in\BB_{\ell+1}$ is a child of $\beta_\ell\in\BB_\ell$ if the local knot vector of $\beta_{\ell+1}$, which is denoted by $\Xi_{\beta_{\ell+1}}$, is a subsequence of $\Xi_{\beta_\ell}^{(\ell+1)}$. In other words, the children of $\beta_{\ell}$ are the B-splines in $\BB_{\ell+1}$ whose local knot vector consist of $p+2$ consecutive knot of $\Xi_{\beta_\ell}^{(\ell+1)}$. We let
 $$\CC(\beta_\ell):=\{\beta_{\ell+1}\in\BB_{\ell+1}\,|\, \beta_{\ell+1} \text{ is a child of }\beta_\ell\}.$$ 

 Conversely, if $\beta_{\ell+1}\in\BB_{\ell+1}$ is given, we define the set of parents of $\beta_{\ell+1}$ by
\begin{equation*}
\PPP(\beta_{\ell+1}):=\{\beta_{\ell}\in\BB_{\ell}\,|\, \beta_{\ell+1} \text{ is a child of }\beta_\ell\}.
\end{equation*}
\end{definition}

It is easy to check that the last definition means that $\beta_{\ell+1}\in\BB_{\ell+1}$ is a child of $\beta_\ell\in\BB_\ell$ if and only if
\begin{enumerate}
 \item[(i)] $\min \Xi_{\beta_{\ell}}\le \min \Xi_{\beta_{\ell+1}} \le \max  \Xi_{\beta_{\ell+1}} \le \max  \Xi_{\beta_{\ell}}$.
 \item[(ii)] If $\xi\in\Xi_{\beta_{\ell+1}}$ matches any of the end points of $\Xi_{\beta_{\ell}}$, then the multiplicity of $\xi$ in $\Xi_{\beta_{\ell+1}}$ is less or equal to the  multiplicity of $\xi$ in $\Xi_{\beta_{\ell}}$.
\end{enumerate}
In particular, notice that if $\beta_{\ell+1}$ is a child of $\beta_\ell$, then $\supp\beta_{\ell+1}\subset\supp\beta_\ell$.

Since $\Xi_\ell\subset \Xi_{\ell+1}$, using the so-called knot insertion formula, all B-splines of level $\ell$ can be written as a linear combination of B-splines of level $\ell+1$. More precisely, if $\beta_\ell\in\BB_\ell$, then,
\begin{equation}\label{E:two scale relation 1d}
\beta_\ell = \sum_{\substack{\beta_{\ell+1}\in\CC(\beta_\ell)}} 
c_{\beta_{\ell+1}}(\beta_\ell)\beta_{\ell+1},
\end{equation}
where the coefficients $c_{\beta_{\ell+1}}(\beta_\ell)$ are positive, and $\CC(\beta_\ell)\subset\BB_{\ell+1}$ is the set of children of $\beta_\ell$.

\paragraph{Tensor-product B-spline bases}
Let $d\ge 1$. In order to define a tensor-product $d$-variate spline function space on $\Omega:=[0,1]^d\subset \RR^d$, we consider ${\bf p}:=(p_1,p_2,\dots,p_d)$ the vector of polynomial degrees with respect to each coordinate direction and ${\bf n}:=(n_1,n_2,\dots,n_d)$, where $n_i\ge p_i+1$. For $i=1,2,\dots,d$, let $\Xi_{p_i,n_i}:=\{\xi_j^{(i)}\}_{j=1}^{n_i+p_i+1}$ be a $p_i$-open knot vector, i.e.,
\begin{equation*}
 0=\xi_1^{(i)}=\dots=\xi_{p_i+1}^{(i)}<\xi_{p_i+2}^{(i)}\le\dots\le\xi_{n_i}^{(i)}<\xi_{n_i+1}^{(i)}=\dots=\xi_{n_i+p_i+1}^{(i)}=1,
\end{equation*}
where the two extreme knots are repeated $p_i+1$ times and any internal knot can be repeated at most $p_i+1$ times. We denote by $\VV_{\bf p,n}$ the tensor-product spline space spanned by the B-spline basis $\BB_{\bf p,n}$  defined as the tensor-product of the univariate B-spline bases $\BB(\Xi_{p_1,n_1}),\ldots,\BB(\Xi_{p_d,n_d})$. Let $\QQ_{\bf p,n}$ be tensor-product mesh consisting of the elements $Q= I_1\times\dots\times I_d$, where $I_i$ is an element (closed interval) of the $i$-th univariate mesh $\II(\Xi_{p_i,n_i})$, for $i=1,\dots,d$. 
 
As we did for the univariate case, we consider a given sequence $\{\VV_\ell\}_{n\in\NN_0}$ of tensor-product $d$-variate spline spaces such that
\begin{equation}\label{E:tensor-product spaces}
 \VV_0\subset \VV_1\subset \VV_2\subset\VV_3\subset\dots,
\end{equation}
with the corresponding tensor-product B-spline bases denoted by
\begin{equation*}
 \BB_0,  \BB_1,  \BB_2, \BB_3,\dots,
\end{equation*}
respectively. More precisely, if ${\bf p}=(p_1,p_2,\dots,p_d)$ is given, for $\ell\in\NN_0$, $\VV_\ell:=\VV_{{\bf p},{\bf n}_\ell}$ is the tensor-product spline space and $\BB_\ell:= \BB_{{\bf p},{\bf n}_\ell}$ is the corresponding B-spline basis, for some ${\bf n}_\ell=(n_1^\ell,n_2^\ell,\dots,n_d^\ell)$. In order to guarantee~\eqref{E:tensor-product spaces}, we assume that if $\xi$ is a knot in $\Xi_{p_i,n_i^\ell}$ with multiplicity $m$, then $\xi$ is also a knot in $\Xi_{p_i,n_i^{\ell+1}}$ with multiplicity at least $m$, for $i=1,\dots,d$ and $\ell\in\NN_0$. Furthermore, we denote by $\QQ_\ell:=\QQ_{{\bf p},{\bf n}_\ell}$ the tensor-product mesh and we say that $Q\in\QQ_\ell$ is a \emph{cell of level $\ell$}. We state some well-known properties of the B-spline basis functions that will be useful in this presentation~\cite{deBoor,Schumaker}:
\begin{itemize}
\item \emph{Local linear independence.} For any nonempty open set $O\subset\Omega$, the functions in $\BB_\ell$ that do not vanish identically on $O$, are linearly independent on $O$.
\item \emph{Positive partition of unity.} The B-spline basis functions of level $\ell$ form a partition of the unity 
on $\Omega$, i.e.,
 \begin{equation}\label{E:partition of unity in B_0}
  \sum_{\beta\in\BB_\ell} \beta \equiv 1,\qquad\text{on }\Omega.
 \end{equation}
\item \emph{Two-scale relation between consecutive levels.} The B-splines of level $\ell$ can be written as a linear combination of B-splines of level $\ell+1$. More precisely, 
\begin{equation}\label{E:two scale relation}
 \beta_\ell = \sum_{\beta_{\ell+1}\in\CC(\beta_{\ell})} 
c_{\beta_{\ell+1}}(\beta_\ell)\beta_{\ell+1}, 
\qquad \forall\,\beta_\ell\in\BB_\ell,
\end{equation}
where the coefficients $c_{\beta_{\ell+1}}(\beta_\ell)$ are positive and can be computed using the corresponding coefficients in the univariate two-scale relation~\eqref{E:two scale relation 1d} and Kronecker products.
Here, $\CC(\beta_\ell)$ is the set of children of $\beta_\ell$, and we say that $\beta_{\ell+1}\in\BB_{\ell+1}$ is a child of $\beta_\ell$ if the $i$-th univariate B-spline which defines $\beta_{\ell+1}$ is a child of the $i$-th univariate B-spline defining $\beta_\ell$, for each coordinate direction $i=1,\dots,d$. 
\end{itemize}

\begin{remark}
Notice that if we define $c_{\beta_{\ell+1}}(\beta_\ell):=0$ when $\beta_{\ell+1}$ is not a child of $\beta_\ell$, then equation~\eqref{E:two scale relation} can be written as
\begin{equation}\label{E:two scale relation 2}
 \beta_\ell = \sum_{\beta_{\ell+1}\in\BB_{\ell+1}} 
c_{\beta_{\ell+1}}(\beta_\ell)\beta_{\ell+1}, 
\qquad \forall\,\beta_\ell\in\BB_\ell.
\end{equation}

In particular, we remark that
\begin{equation}\label{E:set of children}
\CC(\beta_\ell)=\{\beta_{\ell+1}\in\BB_{\ell+1}\,|\,c_{\beta_{\ell+1}}(\beta_\ell)>0\}\subset\{\beta_{\ell+1}\in\BB_{\ell+1}\,|\,\supp\beta_{\ell+1}\subset\supp\beta_\ell\}.
\end{equation}
Finally, we also consider the set of parents of $\beta_{\ell+1}\in\BB_{\ell+1}$ given by
\begin{equation*}
\PPP(\beta_{\ell+1}):=\{\beta_{\ell}\in\BB_{\ell}\,|\, \beta_{\ell+1} \text{ is a child of }\beta_\ell\}.
\end{equation*}

\end{remark}

\section{Hierarchical B-spline basis}\label{S:hierarchical basis}

\begin{definition} 
If $n\in\NN$, 
we say that ${\bf\Omega}_n := \{\Omega_0,\Omega_1,\dots,\Omega_n\}$ is a 
\emph{hierarchy of subdomains of $\Omega$ of depth $n$} if
\begin{enumerate}
 \item[(i)] $\Omega_\ell$ is the union of cells of level 
$\ell-1$, for $\ell = 1,2,\dots,n$.
 \item[(ii)]$\Omega = \Omega_0 \supset \Omega_1 \supset \dots \supset 
\Omega_{n-1}\supset \Omega_n = \emptyset$.
\end{enumerate} 
\end{definition}

We now define the hierarchical B-spline basis ${\HH}= {\HH}({\bf\Omega}_n)$ in the following recursive way
\begin{equation}\label{E:definition of Hltilde}
\begin{cases}\HH_0:=\BB_0,\\
 {\HH}_{\ell+1}:=\{\beta\in\HH_{\ell}\,|\,\supp\beta\not\subset\Omega_{\ell+1}\}\cup\{\beta\in\BB_{\ell+1}\,|\, \supp\beta\subset\Omega_{\ell+1}\},\qquad\ell=0,\dots,n-2.\\
 \HH:=\HH_{n-1}.\end{cases}
\end{equation}
If $\RRR_\ell:=\HH_\ell\setminus\HH_{\ell+1}$, then
$$\RRR_\ell=\{\beta\in\HH_{\ell}\,|\,\supp\beta\subset\Omega_{\ell+1}\}=\{\beta\in\BB_{\ell}\,|\,\supp\beta\subset\Omega_{\ell+1}\}.$$
Notice that in order to get $\HH_{\ell+1}$ from $\HH_{\ell}$ we replace the set $\RRR_\ell$ by
$$\{\beta\in\BB_{\ell+1}\,|\, \supp\beta\subset\Omega_{\ell+1}\}.$$
Moreover, it is easy to check that
\begin{equation}\label{E:Hierarchical basis}
{\HH}= \bigcup_{\ell = 0}^{n-1} \{\beta \in \BB_\ell \mid \supp \beta 
\subset \Omega_\ell \wedge  \supp \beta \not\subset \Omega_{\ell+1}\}.
\end{equation} 
We say that $\beta$ is \emph{active} if $\beta\in\HH$. The 
corresponding underlying mesh $\QQ = \QQ({\bf\Omega}_n)$ is given 
by
\begin{equation}\label{E:hierarchical mesh}
\QQ:= \bigcup_{\ell = 0}^{n-1} \{ Q\in\QQ_\ell \,\mid\, Q\subset \Omega_\ell 
\,\wedge\, Q\not\subset \Omega_{\ell+1}\}, 
\end{equation}
and we say that $Q$ is an \emph{active cell} is $Q\in\QQ$, or that $Q$ is an 
\emph{active cell of level $\ell$} if $Q\in\QQ\cap \QQ_\ell$.

Our definition of the hierarchical B-spline basis $\HH$ is slightly different from the one given in~\cite{Kraft-thesis,Vuong etal}, because they consider the domains $\Omega_\ell$, the cells $Q$ and the function supports as open sets. That definition does not allow to rebuild the hierarchical space uniquely from the only knowledge of the hierarchical mesh (i.e., the active cells of each level); see for example~\cite[Figure 2.5]{Kraft-thesis} or~\cite[Figure 2]{Kraft}. 

We know (cf.~\cite{Vuong etal}) that functions in $\HH$ constitutes a linearly independent set and that 
$$\VV_0=\Span\BB_0\subset\Span\HH.$$

Unlike the B-spline bases $\BB_\ell$ for tensor-product spline spaces, the hierarchical B-spline basis $\HH$ does not constitute a partition of the unity. Instead, we can prove the following result. 

\begin{lemma}\label{L:definition of the weights}(Partition of unity in $\HH$). Let $\HH$ be the hierarchical B-spline basis associated to the hierarchy of 
subdomains of depth $n$,
${\bf\Omega}_n := \{\Omega_0,\Omega_1,\dots,\Omega_n\}$. Let $a_{\beta_0} := 1$ for all $\beta_0\in\BB_0$ and
\begin{equation}\label{E:definition of the weights}
 a_{\beta_{\ell+1}}:=\sum_{\substack{\beta_\ell\in\BB_\ell \\ \supp 
\beta_\ell\subset\Omega_{\ell+1}}} a_{\beta_\ell} 
c_{\beta_{\ell+1}}(\beta_\ell),\qquad\forall\,\beta_{\ell+1}\in\BB_{\ell+1},\,
\supp\beta_{\ell+1}\subset\Omega_{\ell+1},
\end{equation}
for $\ell = 0,1,\dots,n-2.$ Then,
\begin{equation*}
\sum_{\beta\in {\HH}} a_\beta \beta \equiv 1, \qquad\text{on } \Omega. 
\end{equation*}
\end{lemma}

\begin{remark}
 In view of the linear independence of functions in $\HH$, we have that the set $\{a_\beta\}_{\beta\in\HH}$ is uniquely determined. On the other hand, we remark that the definition~\eqref{E:definition of the weights} depends on the hierarchy of subdomains~${\bf\Omega}_n$, and that the coefficients $a_\beta$ are defined not only for $\beta\in\HH$ but also for $\beta\in\bigcup_{\ell=0}^{n-1}\{\beta_\ell\in\BB_\ell\,|\,\supp\beta_\ell\subset\Omega_\ell\}.$
\end{remark}

\begin{proof}
By~\eqref{E:partition of unity in B_0} we have that $\sum_{\beta\in \HH_0} 
a_\beta \beta \equiv 1$, on $\Omega$. Assume now that for a fixed $\ell$ 
($0\le \ell\le n-2$), $\sum_{\beta\in \HH_\ell} a_\beta \beta \equiv 1$, on 
$\Omega$. Thus, using~\eqref{E:definition of the weights} and~\eqref{E:two scale relation 2} we have that
\begin{align*}
 \sum_{\beta\in \HH_{\ell+1}} a_\beta \beta &= \sum_{\substack{\beta\in\HH_\ell \\ \supp 
\beta\not\subset\Omega_{\ell+1}}} a_\beta \beta +
 \sum_{\substack{\beta_{\ell+1}\in  \BB_{\ell+1}\\ \supp\beta_{\ell+1}\subset 
\Omega_{\ell+1}}} a_{\beta_{\ell+1}} \beta_{\ell+1}\\ &= \sum_{\substack{\beta\in\HH_\ell \\ \supp 
\beta\not\subset\Omega_{\ell+1}}} a_\beta \beta 
+\sum_{\substack{\beta_\ell\in\BB_\ell \\ \supp 
\beta_\ell\subset\Omega_{\ell+1}}} a_{\beta_\ell} \beta_\ell\\
&=\sum_{\beta\in \HH_\ell} a_\beta \beta.
\end{align*}
Thus, $\sum_{\beta\in \HH_{\ell+1}} a_\beta \beta \equiv 1$, on $\Omega$. 

Finally, the proof is complete regarding that $\HH=\HH_{n-1}$.
\end{proof}

 Notice that in the last lemma $a_\beta\ge 0$, for all $\beta\in\HH$. The next result characterizes the functions $\beta$ whose weight $a_\beta$ is equal to zero.

\begin{theorem}\label{T:caracterization of nonzero weights}
 Let $\beta_{\ell+1}\in\BB_{\ell+1}$ be such that 
$\supp\beta_{\ell+1}\subset\Omega_{\ell+1}$, for some 
$\ell=0,1,\dots,n-2$. Then, the following statements are equivalent:
\begin{enumerate}
 \item [(i)] $a_{\beta_{\ell+1}} = 0$.
 \item [(ii)] 
$\forall\,\beta_\ell\in\PPP(\beta_{\ell+1}),\quad (a_{\beta_\ell}>0 \,\,\wedge\,\, \supp\beta_{\ell}\subset\Omega_{\ell}\quad\Longrightarrow\quad \supp\beta_\ell\not\subset\Omega_{\ell+1})$.
 \end{enumerate}
\end{theorem}

\begin{proof}
 From~\eqref{E:definition of the weights}, we have that $a_{\beta_{\ell+1}} =0$ 
if and only if for all $\beta_{\ell}\in\BB_\ell$, with
$\supp\beta_{\ell}\subset\Omega_{\ell+1}$,
$$a_{\beta_{\ell}}=0 \qquad\text{ or } \qquad c_{\beta_{\ell+1}}(\beta_\ell) = 0,$$
or, equivalently, if and only if, all parents $\beta_{\ell}$ of $\beta_{\ell+1}$ with $a_{\beta_\ell}>0$, satisfy 
$\supp\beta_{\ell}\not\subset\Omega_{\ell+1}$.
\end{proof}

Notice that there is no function $\beta_0\in\BB_0$ satisfying $a_{\beta_0}=0$. Now, the last theorem implies that $a_{\beta_1}=0$ for $\beta_1\in\BB_1$ with $\supp\beta_1\subset\Omega_1$ if and only if all its parents are active. Roughly speaking, $\Omega_1$ is too narrow around $\supp\beta_1$. Notice that, as soon as one of the $\beta_1$'s parents is deactivated, $a_{\beta_1}$ will become positive in the new configuration.  

We conclude this section with the following result, which states that each deactivated B-spline of level $\ell$ can be written as a linear combination of functions in the hierarchical basis of the subsequent levels, i.e., $\ell+1$,\dots, $n-1$.

\begin{lemma}\label{L:decomposition1}
Let $\HH$ be the hierarchical B-spline basis associated to the hierarchy of 
subdomains of depth $n$,
${\bf\Omega}_n := \{\Omega_0,\Omega_1,\dots,\Omega_n\}$. Then, 
\begin{equation}\label{E:decomposition1}
\RRR_\ell\subset \Span\left(\HH\cap\bigcup_{k=\ell+1}^{n-1}\BB_k\right),  
 \end{equation}
for $\ell=0,1,\dots,n-2$. 
\end{lemma}

\begin{proof}
 Notice that~\eqref{E:decomposition1} holds for $\ell=n-2$ due to~\eqref{E:two scale relation 2}. Let us assume that~\eqref{E:decomposition1} holds for some $\ell$, and prove that it holds for $\ell-1$. Let $\beta_{\ell-1}\in\RRR_{\ell-1}$. Since $\beta_{\ell-1}\in\BB_{\ell-1}$ and $\supp\beta_{\ell-1}\subset\Omega_\ell$, we have that
 $$\beta_{\ell-1} = \sum_{\substack{\beta_\ell\in\BB_\ell\\\supp\beta_\ell\subset \Omega_{\ell}}} c_{\beta_\ell}(\beta_{\ell-1})\beta_\ell=\sum_{\substack{\beta_\ell\in\BB_\ell\\\supp\beta_\ell\subset \Omega_{\ell+1}}} c_{\beta_\ell}(\beta_{\ell-1})\beta_\ell+\sum_{\beta_\ell\in\HH\cap\BB_\ell} c_{\beta_\ell}(\beta_{\ell-1})\beta_\ell.$$
 Thus, $\beta_{\ell-1}\in\Span\left(\HH\cap\bigcup_{k=\ell}^{n-1}\BB_k\right)$, which concludes the proof.
\end{proof}

\section{Approximation properties and quasi-interpolation}\label{S:interpolation}

Let ${\bf\Omega}_n := \{\Omega_0,\Omega_1,\dots,\Omega_n\}$ be a hierarchy of subdomains of $\Omega$ of depth $n$ and let $\QQ$ be the hierarchical mesh given by~\eqref{E:hierarchical mesh}. Let ${\HH}$ be the hierarchical basis defined in~\eqref{E:Hierarchical basis}. 

A multiscale quasi-interpolant operator has been introduced by Kraft~\cite{Kraft-thesis}, where pointwise approximation estimates were established, for the case $\Omega = \RR^2$, where there is no boundary. There, the (infinite) knot vector in each direction considered for building the initial tensor-product space $\VV_0$ was the set of the integer numbers $\ZZ$, and then dyadic refinement is performed to obtain the subsequent levels. In particular, we extend the Kraft construction to the case of open knot vectors with possible multiple internal knots. Moreover, we present a multiscale quasi-interpolant operator which provides suitable local approximation orders in $L^q$-norm, for any $1\le q\le\infty$. 

For each coordinate direction $i=1,\dots,d$, we assume that the sequence of knot vectors $\{\Xi_{p_i,n_i^\ell}\}_{\ell\in\NN_0}$ satisfies:
\begin{itemize}
 \item The sequence of meshes $\{Z_{p_i,n_i^\ell}\}_{\ell\in\NN_0}$ is locally quasi-uniform with parameter $\theta_i>0$.
 We let $\theta :=\max_{i=1,\dots,d} \theta_i.$
 \item The meshsize in each direction is at least halved when moving from a level to the next one, i.e.,
\begin{equation}\label{E:dyadic}
 h_{{\ell+1},i}\le \frac12 h_{\ell,i},
\end{equation}
for all $\ell\in\NN_0$, where $h_{\ell,i}$ is the maximum length of the intervals in $\II(\Xi_{p_i,n_i^\ell})$.
\end{itemize}

\begin{remark}\label{R:Hypotesis on knot vectors}
Notice that if the initial knot vectors for each direction, i.e., $\Xi_{p_1,n_1^0}$, $\Xi_{p_2,n_2^0}$, $\dots$, $\Xi_{p_d,n_d^0}$, are chosen arbitrarily and we perform dyadic refinement for obtaining the subsequent levels, we will obtain a sequence of meshes which satisfies the two conditions just stated. 
\end{remark}

In order to define the multiscale quasi-interpolant operator in the hierarchical space, we need to introduce first some \emph{local} quasi-interpolant operators $P_\ell$, for $\ell=0,1,\dots,n-1$, satisfying certain suitable properties. These last operators can be defined using the ideas from~\cite{LLM01}, where each operator is defined using a underlying local approximation method. 

\paragraph{A local approximation method.} 
We recall that ${\bf p}:=(p_1,p_2,\dots,p_d)$ and denote by $\mathbb{P}_{\bf p}$ the set of tensor-product polynomials with degree at most $p_i$ in the coordinate direction $x_i$, for $i=1,2,\dots,d$. Let $N:=\dim\mathbb{P}_{\bf p}=\Pi_{i=1}^d (p_i+1)$. 

Let $\ell\in\NN_0$ be arbitrary and fixed. For $Q\in\QQ_\ell$ given, we consider the basis $\BB_Q:=\{\beta_1^Q,\dots,\beta_{N}^Q\}$ of $\mathbb{P}_{\bf p}$, consisting of the B-spline basis functions in $\BB_\ell$ which are nonzero on~$Q$.
 
 Let $\Pi_Q:L^1(Q)\to \mathbb{P}_{\bf p}$ be the $L^2$-projection operator defined by
\begin{equation}\label{E:local L2 projection}
 \int_Q (f-\Pi_Q f) g = 0, \qquad\forall\, g\in\mathbb{P}_{\bf p}.
 \end{equation}
 
Notice that 
\begin{equation}\label{E:dual basis}
 \Pi_Q f =\sum_{i=1}^N \lambda_i^Q (f) \beta_i^Q,\qquad\forall\,f\in L^1(Q),
\end{equation}
where $\lambda^Q(f):=(\lambda_1^Q (f),\dots,\lambda_N^Q (f))^T$ is the solution of the linear system
\begin{equation*}
 M_Q {\bf x} = F_Q,
\end{equation*}
where 
\begin{equation*}
 M_Q= \left(\int_Q \beta_j^Q\beta_i^Q\right)_{i,j=1,\dots,N}\in\RR^{N\times N},\quad\text{and}\quad
  F_Q=\left(\int_Q f\beta_i^Q\right)_{i=1,\dots,N}\in\RR^{N\times 1}.
 \end{equation*}
 
Since $\Pi_Q$ preserves polynomials in $\mathbb{P}_{\bf p}$, we have that $\{\lambda_i^Q:L^1(Q)\to \RR\,|\,i=1,\dots,N\}$ is a dual basis for $\BB_Q$ in the the sense that 
\begin{equation}\label{E:Duality for local functionals}
\lambda_i^Q(\beta_j^Q)=\begin{cases}1,\qquad \text{if }i=j\\ 0,\qquad \text{if }i\neq j\end{cases},\qquad i,j=1,\dots,N. 
\end{equation}

As a consequence of the $L^\infty$-local stability of the B-spline basis we have the following result (cf.~\cite{BGGS15}).
\begin{lemma}\label{L:bound for the dual basis}
Let $q$ be such that $1\le q\le \infty$. Let $Q\in\QQ_\ell$ and let $\Pi_Q:L^1(Q)\to \mathbb{P}_{\bf p}$ be the $L^2$-projection operator defined by~\eqref{E:local L2 projection}. Then, there exists a constant $C_{{\bf p},\theta}>0$ which depends only on ${\bf p}$ and $\theta$ such that
 \begin{equation*}
  \|\lambda^Q(f)\|_\infty\le C_{{\bf p},\theta}|Q|^{-\frac1q}\|f\|_{L^q(Q)},\qquad\forall\,f\in L^q(Q),
 \end{equation*}
where $\lambda^Q(f) = (\lambda_1^Q (f),\dots,\lambda_N^Q (f))^T$ are the coefficients of $\Pi_Q(f)$ with respect to the local basis $\BB_Q$ as given~\eqref{E:dual basis}.
\end{lemma}

\paragraph{A locally supported dual basis.}
 For $\ell=0,1,\dots,n-1$, we define $\omega_\ell$ as the union of the elements of level $\ell$ whose support extension is contained in $\Omega_\ell$, i.e.,
\begin{equation*}
\omega_\ell:=\bigcup_{\substack{Q\in\QQ_\ell\\ \tilde Q\subset\Omega_\ell}} Q,
\end{equation*}
where $\tilde Q$ denotes the support extension of $Q$, given by $\tilde Q =\D\bigcup_{\substack{\beta\in\BB_\ell\\\supp\beta\supset Q}}\supp\beta$. In other words, $\omega_\ell$ consists of the elements of level $\ell$ where the full tensor-product space of level $\ell$ can be exactly represented in the hierarchical space.

Let 
\begin{equation}\label{E:B ell omega ell}
 \BB_{\ell,{\omega_\ell}}:=\{\beta\in\BB_\ell\,|\,\exists\,Q\in\QQ_\ell\,:\,Q\subset\supp\beta\cap\omega_\ell\},
\end{equation}
and
\begin{equation*}
\VV_{\omega_\ell}:=\Span\BB_{\ell,{\omega_\ell}}. 
\end{equation*}

We remark that $\BB_{\ell,{\omega_\ell}}\subset \{\beta\in\BB_\ell\,|\,\supp\beta\subset\Omega_\ell\}$, but in general, $\BB_{\ell,{\omega_\ell}}\varsubsetneq \{\beta\in\BB_\ell\,|\,\supp\beta\subset\Omega_\ell\}$, see Figure~\ref{F:example} below. The goal of this paragraph is to define a dual basis for the multivariate B-spline basis $\BB_{\ell,{\omega_\ell}}$, i.e., a set of linear functionals
$$\{\lambda_\beta:L^1(\Omega)\to\RR\,|\,\beta\in\BB_{\ell,{\omega_\ell}}\},$$
such that $\lambda_{\beta_i}(\beta_j) = \begin{cases}1,\qquad \text{if }i=j\\ 0,\qquad \text{if }i\neq j\end{cases}$, for all $\beta_i,\beta_j\in\BB_{\ell,{\omega_\ell}}$.  We will use the technique presented in~\cite{LLM01} together with the local $L^2$-projection defined in~\eqref{E:local L2 projection}. Roughly speaking, we define the functional $\lambda_\beta$ as a local projection onto some $Q_\beta\in\QQ_\ell$ such that $Q_\beta\subset\supp\beta\cap\omega_\ell$. More precisely, for each $\beta\in\BB_{\ell,{\omega_\ell}} $, we choose $Q_\beta\in\QQ_\ell$ such that $Q_\beta\subset\supp\beta\cap\omega_\ell$ and let 
$$\lambda_\beta:=\lambda_{i_0}^{Q_\beta},$$ where $i_0=i_0(\beta,Q_\beta)$ with $1\le i_0\le N$ is such that $\beta_{i_0}^{Q_\beta}\equiv \beta$ on $Q_\beta$. 

As an immediate consequence of~\eqref{E:Duality for local functionals} and Lemma~\ref{L:bound for the dual basis} we have the following result.

\begin{proposition}\label{P:proposition} Let $\{\lambda_\beta:L^1(\Omega)\to\RR\,|\,\beta\in\BB_{\ell,{\omega_\ell}}\}$ be the set of linear functionals just defined above. Then, the following properties hold:
\begin{enumerate}[(i)]
 \item \emph{Local support:} If $Q_\beta$ denotes the element in $\QQ_\ell$ chosen for the definition of $\lambda_\beta$, then $\lambda_\beta$ is \emph{supported} in $Q_\beta$, i.e.,
\begin{equation*}
 \forall\,f\in L^q(\Omega),\quad f_{|_{Q_\beta}} \equiv 0\qquad\Longrightarrow\qquad \lambda_\beta(f) = 0.
\end{equation*}
\item \emph{Dual basis:} For $\beta_i,\beta_j\in\BB_{\ell,{\omega_\ell}}$, $\lambda_{\beta_i}(\beta_j)=\begin{cases}1,\qquad \text{if }i=j\\ 0,\qquad \text{if }i\neq j\end{cases}$. 
\item \emph{$L^q$-Stability:} If $\beta\in\BB_{\ell,{\omega_\ell}}$ and $f\in L^q(Q_\beta)$,
               \begin{equation}\label{E:Stability of lambda tau}                             
                            |\lambda_\beta(f)|\le C_{{\bf p},\theta} |Q_\beta|^{-\frac1q}\|f\|_{L^q(Q_\beta)}.\end{equation}
                             
\end{enumerate}                       
\end{proposition}

\paragraph{Localized quasi-interpolant operators in the tensor-product spaces.} 

Now, we are in position of defining a quasi-interpolant operator for each level $\ell$, using the dual bases introduced in the previous paragraph. For $\ell =0,1,\dots,n-1$, let $P_\ell: L^q(\Omega)\to\Span\BB_{\ell,{\omega_\ell}}\subset \VV_\ell$ be given by
\begin{equation}\label{E:quasi-interpolant l}
P_\ell f:=\sum_{\beta\in \BB_{\ell,{\omega_\ell}}}\lambda_\beta(f)\beta,\qquad \forall\,f\in L^q(\Omega).
\end{equation}
           
The next result summarizes the main properties of $P_\ell$.

\begin{theorem}\label{T:properties of Pl}
For $\ell =0,1,\dots,n-1$, let $P_\ell$ be the operator given by~\eqref{E:quasi-interpolant l}. Then, the following properties hold:
\begin{enumerate}[(i)]
 \item $P_\ell$ preserves splines in $\VV_{\omega_\ell}$, i.e., $P_\ell s = s$, for all $s\in\VV_{\omega_\ell}.$

 \item $P_\ell$ is supported in $\omega_\ell$, i.e., 
 \begin{equation}\label{E:Pl supported in omega_l}
 \forall\,f\in L^q(\Omega),\quad f_{|_{\omega_\ell}} \equiv 0\qquad\Longrightarrow\qquad P_\ell f \equiv 0.
\end{equation}

\item For all $s\in\VV_\ell$, 
$$P_\ell s\equiv s, \qquad\text{ on }\omega_\ell.$$
 
 \item \emph{Stablility}: The quasi-interpolant operator $P_\ell$ satisfies 
\begin{equation}\label{E:Stability of Pk}
\|P_\ell f\|_{L^q(\Omega_\ell)}\le C_S\|f\|_{L^q(\omega_\ell)}, \qquad\forall\,f\in L^q(\omega_\ell),
\end{equation}
where the constant $C_S>0$ only depends on ${\bf p}$ and $\theta$.
\item \emph{Approximation}: Let ${\bf s}:=(s_1,s_2,\dots,s_d)$ be such that $s_i\le p_i+1$, for $i=1,2,\dots,d$. For $f\in L^{\bf s}_q(\Omega):=\{g\in L^1_{\loc}(\Omega)\,|\, D^{r_i}_{x_i} g\in L^q(\Omega),\,0\le r_i\le s_i,\,i=1,\dots,d\}$,
\begin{equation}\label{E:Approximation of Pk}
\|f-P_\ell f\|_{L^q(\omega_\ell)}\le C_A \sum_{i=1}^d h_{\ell,i}^{s_i}\|D^{s_i}_{x_i} f\|_{L^q(\Omega_\ell)}, 
\end{equation}
 where the constant $C_A>0$ depends on $d$, ${\bf s}$, ${\bf p}$ and $\theta$.
\end{enumerate}
 
\end{theorem}

\begin{proof}
(i) This is an immediate consequence of Proposition~\ref{P:proposition} (ii).

\medskip

(ii) This follows from~\eqref{E:Stability of lambda tau} and~\eqref{E:quasi-interpolant l}.

\medskip

(iii) This is a consequence of (i) and (ii).

\medskip

 (iv) Let $Q\in\QQ_\ell$ such that $Q\subset\Omega_\ell$. Taking into account~\eqref{E:Stability of lambda tau} and~\eqref{E:partition of unity in B_0} we have that
\begin{equation*}
 |P_\ell f|\le\max_{\substack{\beta\in\BB_\ell\\\supp\beta\supset Q}}|\lambda_\beta(f)|\le C |Q|^{-\frac1q}\|f\|_{L^q(\tilde Q\cap\omega_\ell)},\qquad\text{on }Q,
\end{equation*}
for all $f\in L^q(\omega_\ell)$, where $\tilde Q =\D\bigcup_{\substack{\beta\in\BB_\ell\\\supp\beta\supset Q}}\supp\beta$. Here, the constant $C>0$ depends on ${\bf p}$ and $\theta$. Then,
 \begin{equation}\label{E:Stability aux}
\|P_\ell f \|_{L^q(Q)}\le C \|f\|_{L^q(\tilde Q\cap\omega_\ell)}.  
 \end{equation}
 Now,~\eqref{E:Stability of Pk} follows from the last equation. 
 
 \medskip
 
 (v) 
Let $Q\in\QQ_\ell$ such that $Q\subset\omega_\ell$. By results on multidimensional Taylor expansions, there exists $p_{\tilde Q}\in\mathbb{P}_{\bf p}$ such that
\begin{equation}\label{E:Taylor}
 \|f-p_{\tilde Q}\|_{L^q(\tilde Q)}\le C_T \sum_{i=1}^d h_{\ell,i}^{s_i}\|D^{s_i}_{x_i} f\|_{L^q(\tilde Q)},
\end{equation}
where the constant $C_T>0$ only depends on $d$, ${\bf s}$, ${\bf p}$ and $\theta$. Taking into account~\eqref{E:Stability aux}, (iii) and~\eqref{E:Taylor} we have that

 \begin{align*}
       \|f-P_\ell f\|_{L^q(Q)}&\le \|f-p_{\tilde Q}\|_{L^q(Q)}+\|p_{\tilde Q}-P_\ell f\|_{L^q(Q)}\\
       &= \|f-p_{\tilde Q}\|_{L^q(Q)}+\|P_\ell(p_{\tilde Q}- f)\|_{L^q(Q)}\\
       &\le (1+C)\|f-p_{\tilde Q}\|_{L^q(\tilde Q)}\\
       &\le (1+C)C_T \sum_{i=1}^d h_{\ell,i}^{s_i}\|D^{s_i}_{x_i} f\|_{L^q(\tilde Q)}.
      \end{align*}
Now,~\eqref{E:Approximation of Pk} follows.
\end{proof}

\paragraph{A multiscale quasi-interpolant operator.}

Let $P_\ell$ be the operator given by~\eqref{E:quasi-interpolant l}, for each $\ell=0,1,\dots,n-1$. We define $\qih:L^q(\Omega)\to\Span\HH$ by
\begin{equation}\label{E:definition qihh}
\begin{cases}\qih_0:=P_0,\\
 {\qih}_{\ell+1}:= \qih_\ell+P_{\ell+1}(\id-\qih_{\ell}),\qquad\ell=0,\dots,n-2.\\
 \qih:=\qih_{n-1}.\end{cases}
\end{equation}

\begin{remark}
Notice that Theorem~\ref{T:properties of Pl} (i) implies that $P_0s=s$, for all $s\in\VV_0$. Thus, as an immediate consequence of the definition of $\qih$ given in~\eqref{E:definition qihh}, we have that 
\begin{equation*}
 \qih s = s, \qquad\forall\,s\in\VV_0,
\end{equation*}
i.e., $\qih$ preserves splines in the initial level and in particular, tensor-product polynomials in ${\mathbb P}_{\bf p}$.
\end{remark}

\begin{remark}\label{R: Image of Pi}
 Taking into account the definitions of $\qih$ and $P_\ell$ given by~\eqref{E:definition qihh} and~\eqref{E:quasi-interpolant l}, respectively, we have that
 \begin{equation*}
  \qih:L^q(\Omega)\to\sum_{\ell=0}^{n-1} \VV_{\omega_\ell}:=\left\{\sum_{\ell=0}^{n-1}s_\ell\,\mid\,s_\ell\in\VV_{\omega_\ell},\quad \ell=0,1,\dots,n-1\right\}. 
 \end{equation*}
 \end{remark}

When restricting to a fixed $\omega_\ell$, the following result allows to write the multiscale quasi-interpolant $\qih$ in terms of the operators $P_\ell$, $P_{\ell+1}$,$\dots$,$P_{n-1}$.

\begin{theorem}\label{T:decomposition of Pi}
 If
 $$\omega_{n-1}\subset\omega_{n-2}\subset\dots\subset\omega_2\subset\omega_1\subset\omega_0,$$ then 
 \begin{equation*}
\qih f= P_\ell f+\sum_{k=\ell+1}^{n-1}P_k(f-P_{k-1}f),\qquad\text{on }\omega_\ell,\qquad (\ell=0,1,\dots,n-1), 
 \end{equation*}
 for $f\in L^q(\Omega)$.
\end{theorem}

Before proving Theorem~\ref{T:decomposition of Pi} we state the following elementary result.

\begin{lemma}\label{L:qihh aux}
For $f\in L^q(\Omega)$,
 $$\qih_\ell f=P_\ell f,\qquad\text{on }\omega_\ell,\qquad (\ell=0,1,\dots,n-1).$$
 \end{lemma}

\begin{proof}
Let $1\le \ell\le n-1$ and $f\in L^q(\Omega)$. Since $\qih_{\ell-1}f\in \VV_{\ell-1}\subset\VV_\ell$, by Theorem~\ref{T:properties of Pl} (iii),  we have that $P_\ell\qih_{\ell-1}f = \qih_{\ell-1}f$, on $\omega_\ell$. Now, the definition of $\qih_{\ell}$ yields $\qih_\ell f = \qih_{\ell-1}f+P_{\ell}(f-\qih_{\ell-1}f)=P_\ell f$, on $\omega_\ell$.
\end{proof}

\begin{proof}[Proof of Theorem~\ref{T:decomposition of Pi}]
 From the definition of $\qih$ given by~\eqref{E:definition qihh} we have that
 \begin{equation*}
  \qih  = \qih_\ell +\sum_{k=\ell+1}^{n-1} P_k(\id -\qih_{k-1}).
 \end{equation*}
Now, since $\omega_k\subset\omega_{k-1}$, using that $P_k$ is supported in $\omega_k$ (cf.~\eqref{E:Pl supported in omega_l}) and Lemma~\ref{L:qihh aux} we have that
\begin{equation*}
 \sum_{k=\ell+1}^{n-1} P_k(\id -\qih_{k-1})=\sum_{k=\ell+1}^{n-1} P_k(\id -P_{k-1}),
\end{equation*}
which concludes the proof.
\end{proof}

Now, we state and prove the main result of this section.

\begin{theorem}[Quasi-interpolation in hierarchical spline spaces]\label{T:interpolation in hierarchical spaces}
Assume that
 $$\omega_{n-1}\subset\omega_{n-2}\subset\dots\subset\omega_2\subset\omega_1\subset\omega_0.$$
 Let ${\bf s}:=(s_1,s_2,\dots,s_d)$ be such that $1\le s_i\le p_i+1$, for $i=1,2,\dots,d$. If $\qih:L^q(\Omega)\to\Span\HH$ is the multiscale quasi-interpolant given by~\eqref{E:definition qihh}, then,
 $$\|f-\qih f\|_{L^q(\omega_\ell)}\le C_A(1+2C_S) \sum_{i=1}^d h_{\ell,i}^{s_i}\|D^{s_i}_{x_i} f\|_{L^q(\Omega_\ell)},\qquad \ell=0,1,\dots,n-1,$$
 for $f\in L^{\bf s}_q(\Omega)$.
\end{theorem}

\begin{proof}
 Let $f\in  L^{\bf s}_q(\Omega)$ and let $\ell$ be such that $0\le\ell\le n-1$. Then, using Theorem~\ref{T:decomposition of Pi},~\eqref{E:Stability of Pk}, that $\omega_k\subset\omega_{k-1}$,~\eqref{E:Approximation of Pk} and~\eqref{E:dyadic}, we have that
 \begin{align*}
  \|f-\qih f\|_{L^q(\omega_\ell)}&\le \|f-P_\ell f\|_{L^q(\omega_\ell)}+\sum_{k=\ell+1}^{n-1}\|P_k(f-P_{k-1}f)\|_{L^q(\omega_\ell)}\\
  &= \|f-P_\ell f\|_{L^q(\omega_\ell)}+\sum_{k=\ell+1}^{n-1}\|P_k(f-P_{k-1}f)\|_{L^q(\Omega_k)}\\
  &\le \|f-P_\ell f\|_{L^q(\omega_\ell)}+C_S\sum_{k=\ell+1}^{n-1}\|f-P_{k-1}f\|_{L^q(\omega_k)}\\
  &\le \|f-P_\ell f\|_{L^q(\omega_\ell)}+C_S\sum_{k=\ell+1}^{n-1}\|f-P_{k-1}f\|_{L^q(\omega_{k-1})}\\
  &\le C_A\left(\sum_{i=1}^d h_{\ell,i}^{s_i}\|D^{s_i}_{x_i} f\|_{L^q(\Omega_\ell)}+C_S\sum_{k=\ell+1}^{n-1}\sum_{i=1}^d h_{k-1,i}^{s_i}\|D^{s_i}_{x_i} f\|_{L^q(\Omega_{k-1})}\right)\\
  & \le C_A\sum_{i=1}^d\left(1+ C_S\sum_{k=\ell+1}^{n-1}\frac{1}{2^{(k-1-\ell)s_i}}\right)  h_{\ell,i}^{s_i}\|D^{s_i}_{x_i} f\|_{L^q(\Omega_\ell)}\\
  & \le C_A\sum_{i=1}^d\left(1+ C_S\sum_{k=0}^{\infty}\frac{1}{2^{ks_i}}\right)  h_{\ell,i}^{s_i}\|D^{s_i}_{x_i} f\|_{L^q(\Omega_\ell)}\\
  &= C_A\sum_{i=1}^d\left(1+ \frac{C_S}{1-2^{-s_i}}\right)h_{\ell,i}^{s_i}\|D^{s_i}_{x_i} f\|_{L^q(\Omega_\ell)}.
 \end{align*}
\end{proof}

\begin{remark}
If the hierarchy of subdomains ${\bf\Omega}_n := \{\Omega_0,\Omega_1,\dots,\Omega_n\}$ satisfies
\begin{equation}\label{E:strictly admissible mesh}
 \Omega_{\ell}\subset\omega_{\ell-1}, \qquad\ell=1,\dots,n,
\end{equation}
we say that the mesh $\QQ$ is \emph{strictly admissible} (of class $2$,~cf.~\cite{BuGi15}). In particular, if the mesh if strictly admissible, in view of~\cite[Proposition 20]{GJS14}, we have that the functions in the truncated basis~\cite{GJS12} which take non-zero values on any active cell belong to at most two different levels. 

Notice that if a mesh is strictly admissible then satisfies 
\begin{equation}\label{E:nesting of small omegas}
\omega_{n-1}\subset\omega_{n-2}\subset\dots\subset\omega_2\subset\omega_1\subset\omega_0.    
\end{equation}
On the other hand, in Figure~\ref{F:non admissible meshes} we show some non strictly admissible meshes which satify~\eqref{E:nesting of small omegas}.
\end{remark}

We conclude this section by applying Theorem~\ref{T:interpolation in hierarchical spaces} to the case of strictly admissible meshes. More precisely, we obtain optimal rates of convergence in each level of the hierarchical mesh when considering the asymptotic behavior (cf.~\cite[Example 2]{SM14}).    

\begin{corollary}\label{C:quasi interpolation}
Assume that each level is obtained by dyadic refinement of the elements of the previous one (see Remark~\ref{R:Hypotesis on knot vectors}). If the mesh is strictly admissible (cf.~\eqref{E:strictly admissible mesh}),  then,
\begin{equation*}
 \|f-\qih f\|_{L^q(\Omega_\ell)}\le C \sum_{i=1}^d h_{\ell,i}^{s_i}\|D^{s_i}_{x_i} f\|_{L^q(\Omega_{\ell-1})},\qquad (\ell=1,\dots,n),
\end{equation*}
 for all $f\in L^{\bf s}_q(\Omega)$, where the constant $C>0$ depends only on $d$, ${\bf s}$ and ${\bf p}$.
\end{corollary}

\begin{figure}[H!tbp]
\begin{center}
\includegraphics[width=.33\textwidth]{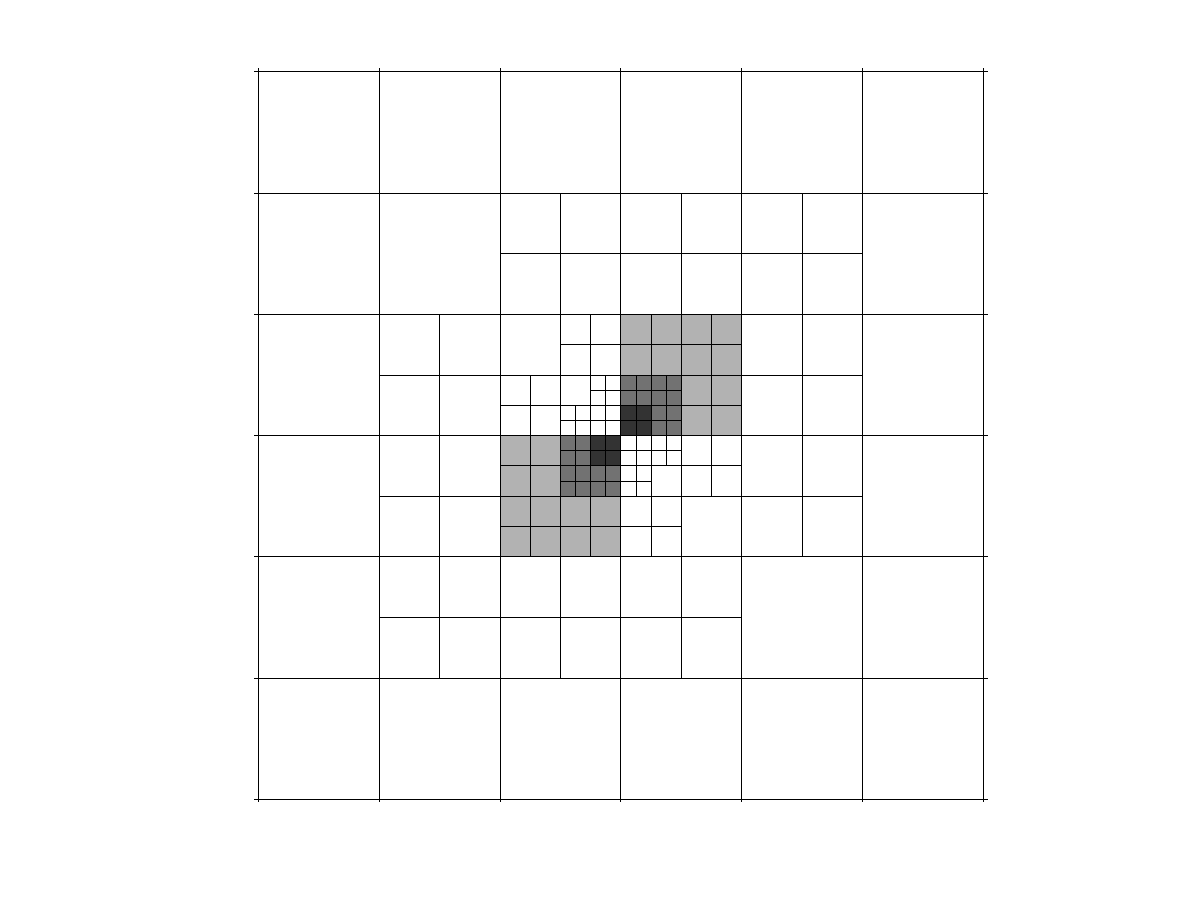} \hfill
\includegraphics[width=.33\textwidth]{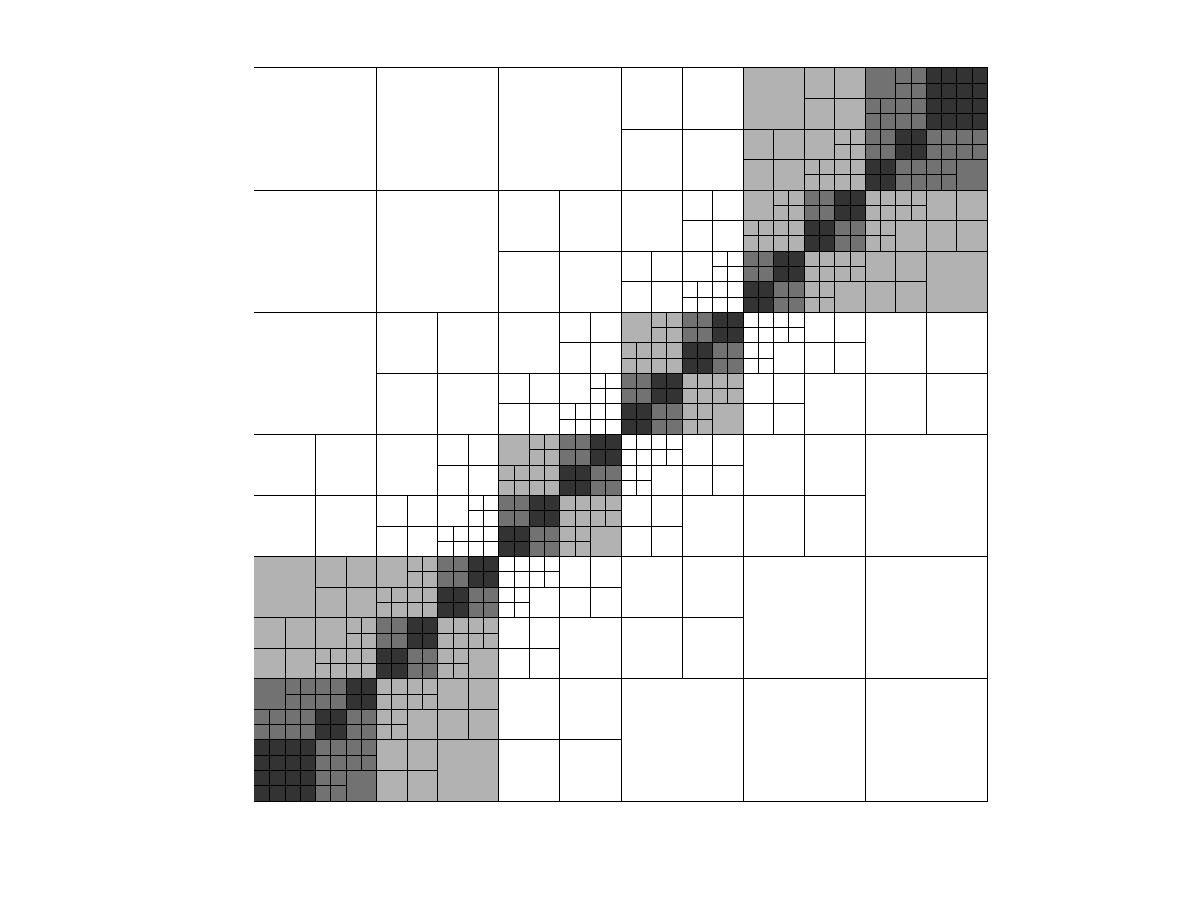}\hfill 
\includegraphics[width=.33\textwidth]{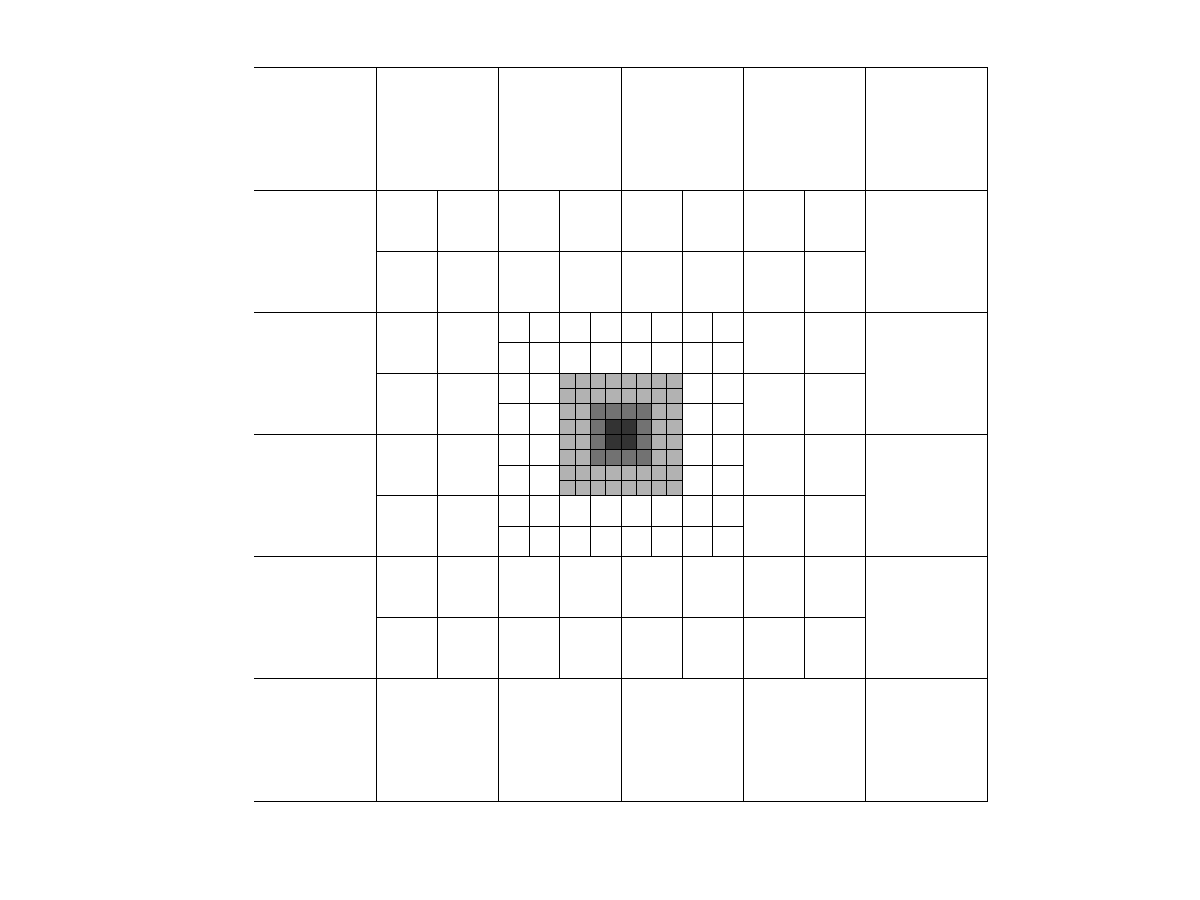}
\end{center}
\caption{\label{F:non admissible meshes} \small Some examples of four-level meshes for splines of maximum smoothness; we consider quadratics ($p=2$) on the left and in the middle, and cubics ($p=3$) on the right. In all cases, the meshes are not strictly admissible but they satisfy~\eqref{E:nesting of small omegas}. The domains $\omega_1$, $\omega_2$ and $\omega_3$ are highlighted in grey from the lightest to the darkest.}
\end{figure}

\section{A new easier hierarchical spline space}\label{S:new hierarchical basis}

Assume that we have already computed the set $\HH_\ell$ for given $\ell$ (cf.~\eqref{E:definition of Hltilde}). Now, in order to compute $\HH_{\ell+1}$ we need to select the new B-splines to be added, i.e., $\{\beta\in\BB_{\ell+1}\,|\, \supp\beta\subset\Omega_{\ell+1}\}$. Once we know the B-spline basis functions to be deactivated $\RRR_\ell=\{\beta\in\BB_{\ell}\,|\,\supp\beta\subset\Omega_{\ell+1}\}$, notice that it is not enough replacing the functions in $\RRR_\ell$ by their children, because in general,
\begin{equation*}
\bigcup_{\beta\in\RRR_\ell}\CC(\beta)\varsubsetneq \{\beta\in\BB_{\ell+1}\,|\, \supp\beta\subset\Omega_{\ell+1}\}.
\end{equation*}

In Figure~\ref{F:example} we show some examples of this situation. This observation suggests a simplified way of selecting B-splines at different levels which consists in adding solely the children of the deactivated functions. Doing so, we obtain a new hierarchical space whose basis that we call $\tilde\HH=\tilde\HH({\bf\Omega}_n)$ is defined as follows:

\begin{figure}[H!tbp]
\begin{center}
\includegraphics[width=.45\textwidth]{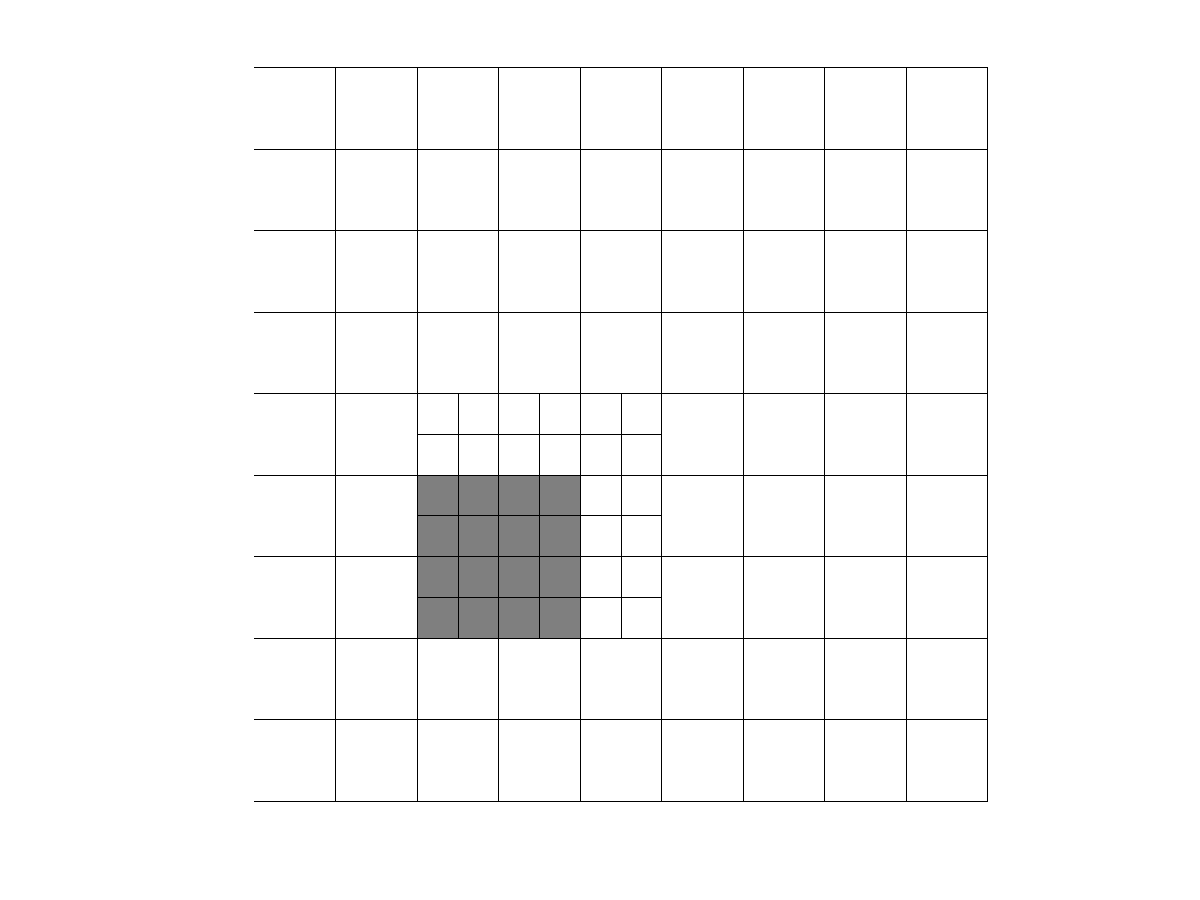} \hfill
\includegraphics[width=.45\textwidth]{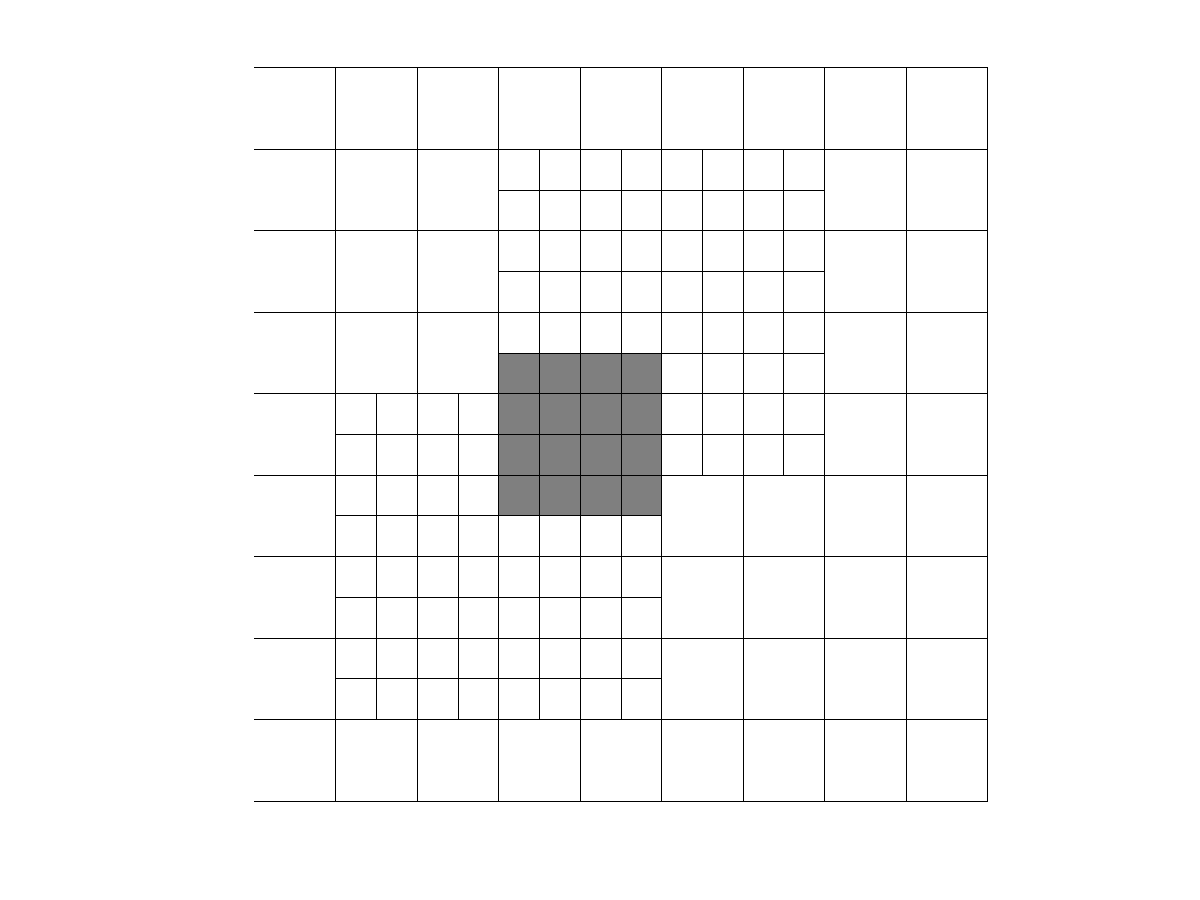} 
\end{center}

\caption{\label{F:example} \small Some examples of two-level meshes for cubic splines ($p=3$) of maximum smoothness. In both cases, the highlighted B-splines of level 1 have support included in $\Omega_1$, but they are not children of any deactivated B-spline of level 0.}
\end{figure}

\begin{equation}\label{E:definition of Hl}
\begin{cases}\tilde\HH_0:=\BB_0,\\
 {\tilde\HH}_{\ell+1}:=\{\beta\in\tilde\HH_{\ell}\,|\,\supp\beta\not\subset\Omega_{\ell+1}\}\cup\D\bigcup_{\substack{\beta\in\tilde\HH_{\ell} \\ \supp\beta\subset\Omega_{\ell+1}}}\CC(\beta),\qquad\ell=0,\dots,n-2.\\
 \tilde\HH:=\tilde\HH_{n-1}.\end{cases}
\end{equation}

In this case, if $\tilde\RRR_\ell:=\tilde\HH_\ell\setminus\tilde\HH_{\ell+1}$, we have that
\begin{equation}\label{E:Rl}
\tilde\RRR_\ell=\{\beta\in\tilde\HH_{\ell}\,|\,\supp\beta\subset\Omega_{\ell+1}\}\subset\{\beta\in\BB_{\ell}\,|\,\supp\beta\subset\Omega_{\ell+1}\}, 
\end{equation}
but now, we can get $\tilde\HH_{\ell+1}$ from $\tilde\HH_{\ell}$ by replacing the B-splines in $\tilde\RRR_\ell$ by their children.

Thus, it seems that building the basis $\tilde\HH$ is easier than the basis $\HH$. In particular, there is no need of traversing the mesh in order to identify the B-splines to add in each recursive step of~\eqref{E:definition of Hl}. However, as an immediate consequence of the following lemma we have that $$\tilde\HH\subset{\HH},$$
and therefore, in general, $\Span\tilde\HH$ can be smaller than $\Span\HH$. 

\begin{lemma}\label{L:Hl subset of tildeHl}
 \begin{equation}\label{E:Hl subset of tildeHl}
  \tilde\HH_\ell\subset\HH_\ell,\qquad\ell=0,1,\dots,n-1.
 \end{equation}
\end{lemma}

\begin{proof}
Notice that~\eqref{E:Hl subset of tildeHl} holds for $\ell=0$ due to $\tilde\HH_0=\HH_0=\BB_0$. Now, using mathematical induction and taking into account~\eqref{E:Rl} and~\eqref{E:set of children} the proof can be completed.
\end{proof}

Thus, since that $\tilde\HH\subset\HH$, when considering the basis $\tilde\HH$ instead $\HH$ for discretizations in isogemetric methods, it will be important to understand which functions we are discarding from the basis and the properties of the space $\Span\tilde\HH$. Regarding the set of coefficients $\{a_\beta\}_{\beta\in\HH}$ for the partition of unity in $\HH$ (cf. Lemma~\ref{L:definition of the weights}) and using Theorem~\ref{T:caracterization of nonzero weights}, we can establish the following characterization for functions in $\tilde\HH$.

\begin{theorem}\label{T:caracterization of H}
For $\ell=0,1,\dots,n-1$,
\begin{equation}\label{E:caracterization for Hl}
\tilde\HH_\ell=\{\beta\in\HH_\ell\mid a_\beta > 0\}. 
\end{equation}
In particular,
 $$\tilde\HH = \{\beta\in {\HH}\mid a_\beta > 0\}.$$
\end{theorem}

\begin{proof}
Since that $a_\beta =1$, for all $\beta\in\BB_0$, we have that~\eqref{E:caracterization for Hl} holds for $\ell=0$. Assume now that~\eqref{E:caracterization for Hl} holds for some $\ell$ and prove that it holds for $\ell+1$. 

Let $\beta\in\tilde\HH_{\ell+1}$. If $\beta\in\tilde\HH_{\ell+1}\cap\tilde\HH_\ell$, using the induction hypothesis we have that $a_\beta>0$. On the other hand, if $\beta\in\tilde\HH_{\ell+1}\setminus\tilde\HH_\ell$, there exists $\beta_\ell\in\tilde\HH_\ell\cap\PPP(\beta)$ such that $\supp\beta_\ell\subset\Omega_{\ell+1}$ and Theorem~\ref{T:caracterization of nonzero weights} yields $a_\beta>0$. Thus, by Lemma~\ref{L:Hl subset of tildeHl} we have that $\tilde\HH_{\ell+1}\subset\{\beta\in\HH_{\ell+1}\mid a_\beta > 0\}$. 

Now, let $\beta\in\HH_{\ell+1}$ satisfying $a_\beta > 0$. If $\beta\in\HH_{\ell+1}\cap\HH_\ell$, using the induction hypothesis we have that $\beta\in\tilde\HH_{\ell+1}$. On the other hand, if $\beta\in\HH_{\ell+1}\setminus\HH_\ell$, we have that $\beta\in\BB_{\ell+1}$ and $\supp\beta\subset\Omega_{\ell+1}$. Thus, using Theorem~\ref{T:caracterization of nonzero weights} we have that there exists $\beta_\ell\in\PPP(\beta)$ such that $a_{\beta_\ell}>0$ and $\supp\beta_\ell\subset\Omega_{\ell+1}$. The induction hypotesis now implies that $\beta_\ell\in\tilde\HH_\ell$ and therefore, $\beta\in\tilde\HH_{\ell+1}$. In consequence,  $\{\beta\in\HH_{\ell+1}\mid a_\beta > 0\}\subset\tilde\HH_{\ell+1}$, which concludes the proof.
\end{proof}

Notice that the functions in $\tilde\HH$ are linearly independent because $\tilde\HH\subset\HH$. On the other hand, from~\eqref{E:definition of Hl} it follows that
\begin{equation}\label{E:nesting}
\Span \tilde\HH_\ell\subset \Span \tilde\HH_{\ell+1},\qquad \ell = 0,1,\dots,n-2,
\end{equation}
and therefore, taking into account that $\tilde\HH_0 = \BB_0$ and $\tilde\HH_{n-1}=\tilde\HH$,
$$\VV_0=\Span \BB_0\subset 
\Span\tilde\HH.$$

\begin{remark}
Since $\Span\BB_0\subset \Span \tilde\HH$, we have that tensor-product polynomials in $\mathbb{P}_{\bf p}$ belong to $\Span\tilde\HH$.
\end{remark}

In the Section~\ref{S:quasi interpolation in H tilde} we study the local approximation properties of the space $\Span\tilde\HH$, through multiscale quasi-interpolant operators.

Finally, taking into account Theorems~\ref{T:caracterization of H} and~\ref{T:caracterization of nonzero weights} we can prove the analogous result of Lemma~\ref{L:decomposition1} when considering the basis $\tilde\HH$.

\begin{lemma}\label{L:decomposition}
Let $\tilde\HH$ be the hierarchical B-spline basis defined by~\eqref{E:definition of Hl} associated to the hierarchy of 
subdomains of depth $n$,
${\bf\Omega}_n := \{\Omega_0,\Omega_1,\dots,\Omega_n\}$. Then, 
 \begin{equation}\label{E:decomposition}
\tilde\RRR_\ell\subset \Span\left(\tilde\HH\cap\bigcup_{k=\ell+1}^{n-1}\BB_k\right),  
 \end{equation}
for $\ell=0,1,\dots,n-2$. 
\end{lemma}

\begin{proof}
 Notice that~\eqref{E:decomposition} holds for $\ell=n-2$ due to~\eqref{E:two scale relation}. Let us assume that~\eqref{E:decomposition} holds for some $\ell$, and prove that it holds for $\ell-1$. Let $\beta_{\ell-1}\in\tilde\RRR_{\ell-1}$. Since $\beta_{\ell-1}\in\BB_{\ell-1}$ and $\supp\beta_{\ell-1}\subset\Omega_\ell$, we have that
 $$\beta_{\ell-1} = \sum_{\beta_\ell\in\CC(\beta_{\ell-1})} c_{\beta_\ell}(\beta_{\ell-1})\beta_\ell=\sum_{\beta_\ell\in\tilde\RRR_\ell} c_{\beta_\ell}(\beta_{\ell-1})\beta_\ell+\sum_{\beta_\ell\in\CC(\beta_{\ell-1})\setminus\tilde\RRR_\ell} c_{\beta_\ell}(\beta_{\ell-1})\beta_\ell.$$
 Thus, $\beta_{\ell-1}\in\Span\left(\tilde\HH\cap\bigcup_{k=\ell}^{n-1}\BB_k\right)$, which concludes the proof.
\end{proof}

\subsection{Quasi-interpolation and local approximation properties in $\Span\tilde\HH$}\label{S:quasi interpolation in H tilde}

In this section we assume that each level is obtained by dyadic refinement of the elements of the previous one (see Remark~\ref{R:Hypotesis on knot vectors}). The following related auxiliary technical result will be useful. The proof is presented in the Appendix 1.

\begin{lemma}\label{L:auxiliar characterization}
 Let $\BB_{\ell,\omega_\ell}$ be the set of B-splines defined in~\eqref{E:B ell omega ell}, for $\ell=0,1,\dots,n-1$. Then,
\begin{equation*}
 \BB_{\ell+1,\omega_{\ell+1}}\subset\bigcup_{\substack{\beta\in\BB_\ell \\ \supp\beta\subset\Omega_{\ell+1}}}\CC(\beta),\qquad \ell = 0,1,\dots,n-2.
\end{equation*}
\end{lemma}

This lemma allows us to prove the following proposition, which, together with the 
results presented in the previous section, show that the space $\Span\tilde \HH$ is \emph{rich enough}, 
and in particular contains all the local spaces $\Span \BB_{\ell,\omega_\ell}$.

\begin{proposition}\label{P:aux quasi-interpolation}
 Assume that
 \begin{equation}\label{E:nesting omega l}
\omega_{n-1}\subset\omega_{n-2}\subset\dots\subset\omega_2\subset\omega_1\subset\omega_0.  
 \end{equation}
  Then, 
  \begin{equation*}
 \BB_{\ell,\omega_{\ell}}\subset  \tilde\HH_\ell,\qquad \ell = 0,1,\dots,n-1.
\end{equation*}
\end{proposition}

\begin{proof}
Taking into account the definition of $\HH_\ell$ given in~\eqref{E:definition of Hltilde} and the characterization of $\tilde\HH_\ell$ in~\eqref{E:caracterization for Hl}, we have that $\{\beta\in\BB_\ell\,|\,\supp\beta\subset\Omega_\ell\,\wedge\, a_\beta>0\}\subset  \tilde\HH_\ell$, for  $\ell = 0,1,\dots,n-1$. On the other hand, for $\beta\in \BB_{\ell,\omega_{\ell}}$, we have that $\supp\beta\subset\Omega_\ell$ and thus, $a_{\beta}$ is well-defined (cf.~\eqref{E:definition of the weights}). Thus, it will be enough to prove that
 \begin{equation}\label{E:aux key}
 \BB_{\ell,\omega_{\ell}}\subset\{\beta\in\BB_\ell\,|\,a_\beta>0\},\qquad \ell = 0,1,\dots,n-1.
\end{equation}
 Notice that~\eqref{E:aux key} holds for $\ell=0$ due to $a_{\beta_0}=1>0$, for all $\beta_0\in\BB_0=\BB_{0,\omega_0}$. Now, assume that~\eqref{E:aux key} holds for some $\ell$. Let $\beta_{\ell+1}\in\BB_{\ell+1,\omega_{\ell+1}}$. 
 In view of Lemma~\ref{L:auxiliar characterization}, there exists $\beta_\ell\in\PPP(\beta_{\ell+1})$ such that $\supp\beta_\ell\subset\Omega_{\ell+1}$. Moreover, taking into account~\eqref{E:nesting omega l}, the definition of $\BB_{\ell+1,\omega_{\ell+1}}$ implies that there exists $Q_{\ell+1}\in\QQ_{\ell+1}$ such that 
 $$Q_{\ell+1}\subset\supp\beta_{\ell+1}\cap\omega_{\ell+1}\subset \supp\beta_\ell\cap\omega_\ell,$$
 which in turn yields $\beta_\ell\in\BB_{\ell,\omega_\ell}$. Finally, taking into account the induction hypotesis, we have that $a_{\beta_\ell}>0$ and now using Theorem~\ref{T:caracterization of nonzero weights} we conclude that $a_{\beta_{\ell+1}}>0$. 
\end{proof}

In view of Remark~\ref{R: Image of Pi}, the immediate consequence of Proposition~\ref{P:aux quasi-interpolation} is that 
the multiscale quasi-interpolant operator defined in Section~\ref{S:interpolation} does construct interpolating 
functions belonging to $\tilde \HH$ and not only to $\HH$. Thus, the approximation estimates 
from Theorem~\ref{T:interpolation in hierarchical spaces} apply verbatim to $\tilde \HH$. We express this fact in the following simple corollary. 

\begin{corollary}Assume that
 $$\omega_{n-1}\subset\omega_{n-2}\subset\dots\subset\omega_2\subset\omega_1\subset\omega_0.$$
 Let $1\le q\le \infty$ and $\qih:L^q(\Omega)\to\Span\HH$ be the multiscale quasi-interpolant operator defined in~\eqref{E:definition qihh}.Then,
 $$\qih:L^q(\Omega)\to \Span \tilde\HH.$$
\end{corollary}


\subsection{Refinement of hierarchical spline spaces}\label{S:refinement}

When thinking of hierarchical splines within a refinement and an adaptation process, it 
is very important to have a precise link between the \emph{enlargement} of the hierarchy of subdomains ${\bf\Omega}_n = \{\Omega_0,\Omega_1,\dots,\Omega_n\}$ and the refinement of the corresponding hierarchical space. 
This issue has been addressed for \emph{classical} hierarchical splines in~\cite{GJS14} and here we address 
it for $\tilde \HH$. 

\begin{definition}
 Let ${\bf\Omega}_n := \{\Omega_0,\Omega_1,\dots,\Omega_n\}$ and ${\bf\Omega}_{n+1}^*:= \{\Omega_0^*,\Omega_1^*,\dots,\Omega_n^*,\Omega_{n+1}^*\}$ be hierarchies of subdomains of $\Omega$ of depth (at most) $n$ and $n+1$, respectively. We say that ${\bf\Omega}_{n+1}^*$ is an \emph{enlargement} of ${\bf\Omega}_n$ if
 $$\Omega_\ell\subset\Omega_\ell^*,\qquad \ell=1,2,\dots,n.$$ 
 \end{definition}

Let ${\bf\Omega}_{n+1}^*$ be an enlargement of ${\bf\Omega}_n$. Now, the corresponding hierarchical B-spline basis ${\HH}^*$ and 
\emph{refined} 
mesh $\QQ^*$ are given by 
\begin{equation*}
{\HH}^*:= \bigcup_{\ell = 0}^{n} \{\beta \in \BB_\ell \mid \supp \beta 
\subset \Omega_\ell^* \wedge  \supp \beta \not\subset \Omega_{\ell+1}^*\},
\end{equation*}
and
$$\QQ^*:= \bigcup_{\ell = 0}^{n}\{ Q\in\QQ_\ell \,\mid\, Q\subset \Omega_\ell^* 
\,\wedge\, Q\not\subset \Omega_{\ell+1}^*\}.$$

In~\cite{GJS14} has been proved that 
\begin{equation}\label{E:nesting by refinement}
\Span\HH\subset\Span\HH^*. 
\end{equation}

Let $\{a_\beta^*\}_{\beta\in{\HH}^*}$ denote the sequence of coefficients (with respect to the hierarchy ${\bf\Omega}_{n+1}^*$) given by Lemma~\ref{L:definition of the weights}. Thus, we have that $$\sum_{\beta\in{\HH}^*} a_\beta^*\beta \equiv 1,\qquad\text{ on }\Omega,$$ and thus, we can consider
$$\tilde\HH^* := \{\beta\in {\HH}^*\mid a_\beta^* > 0\}.$$

The following theorem establishes the analogous of~\eqref{E:nesting by refinement} when considering the basis~$\tilde\HH$.

\begin{theorem}\label{T:nesting by refinement}
$$\Span\tilde\HH\subset \Span\tilde\HH^*.$$
\end{theorem}

In order to prove this result we need the following auxiliary lemma.

\begin{lemma}\label{L:behaviour of weights by refinement}
If $\beta\in\BB_\ell$ and $\supp\beta\subset\Omega_\ell$, for $\ell=0,1,\dots,n-1$, then
$$a_{\beta}^*\ge a_\beta.$$
\end{lemma}

\begin{proof}
 The assertion holds for $\ell=0$ due to $a_\beta^*=a_\beta = 1$ for all $\beta\in\BB_0$. Now, assume that for some $\ell$ we have that
 \begin{equation}\label{E:induction hypothesis}
a_{\beta_\ell}^*\ge a_{\beta_\ell},\qquad \text{for }\beta_\ell\in\BB_\ell,\text{ such that }\supp\beta_\ell\subset\Omega_\ell.  
 \end{equation}
 Let $\beta_{\ell+1}\in\BB_{\ell+1}$ such that $\supp\beta_{\ell+1}\subset\Omega_{\ell+1}$. Since $\Omega_{\ell+1}\subset\Omega_{\ell+1}^*$, using~\eqref{E:induction hypothesis} and the definitions of $a_{\beta_{\ell+1}}$ and $a_{\beta_{\ell+1}}^*$, we have that
 $$a_{\beta_{\ell+1}}^*=\sum_{\substack{\beta_\ell\in\BB_\ell \\ \supp 
\beta_\ell\subset\Omega_{\ell+1}^*}} a_{\beta_\ell}^* 
c_{\beta_{\ell+1}}(\beta_\ell)\ge 
\sum_{\substack{\beta_\ell\in\BB_\ell \\ \supp 
\beta_\ell\subset\Omega_{\ell+1}}} a_{\beta_\ell}^* 
c_{\beta_{\ell+1}}(\beta_\ell)\ge\sum_{\substack{\beta_\ell\in\BB_\ell \\ \supp 
\beta_\ell\subset\Omega_{\ell+1}}} a_{\beta_\ell} 
c_{\beta_{\ell+1}}(\beta_\ell)=a_{\beta_{\ell+1}}.$$
\end{proof}

\begin{proof}[Proof of Theorem~\ref{T:nesting by refinement}]
Let $\beta\in\tilde\HH$ and let $\ell$ be such that $\beta\in\BB_\ell$. Since $\supp\beta\subset\Omega_\ell\subset\Omega_\ell^*$, we have that $\beta\in\HH_\ell^*$. On the other hand, Lemma~\ref{L:behaviour of weights by refinement} implies that $a_\beta^*\ge a_\beta>0$, and~\eqref{E:caracterization for Hl} yields $\beta\in\tilde\HH_\ell^*$. Finally, taking into account~\eqref{E:nesting} we have that $\beta\in\Span\tilde\HH^*$.
\end{proof}

\section{Concluding remarks}\label{S:remarks}

In this paper, after studying the approximation properties of hierarchical splines as defined in~\cite{Kraft-thesis},  we propose a  new hierarchical 
spline space, through a construction of a set of basis functions named $\tilde\HH$, that  enjoys several properties and may be considered as a valuable alternative to truncated hierarchical splines~\cite{GJS12}.  We can summarize and comment upon our results as follows:

\begin{itemize}
\item The basis that we construct simplifies the implementation and data structures needed to 
carry hierarchical splines because the refinement can be performed through the parent-children relations between B-splines. 
Unlike the classical hierarchical space, where algorithms traversing the mesh are needed to identify the new active B-splines, we just add children of already active B-splines. Moreover, we believe that our construction can be suitably used in conjunction with function-based error indicators, i.e., error indicators that mark functions (and not elements) to be refined: in our framework, a marked function would be simply replaced by some of its children. These aspects are studied in the forthcoming paper~\cite{BuGa15}.
\item The weighted basis $\{a_\beta\beta\,|\,\beta\in\tilde\HH\}$ constitutes a convex partition of unity and has the advantage of preserving simple basis function supports (only hypercubes) and also, in principle, ask for the use of simple spline evaluation formulae. 
\item We have extended the multiscale quasi-interpolant proposed in~\cite{Kraft-thesis}  to the cases of general and open knot vectors and we have provided local approximation estimates in $L^q$-norms, for $1\le q\le \infty$. This interpolant is  built in a multiscale fashion, it is not a projector in general and is not based on dual functionals, unlike the one presented in~\cite{SM14} based on THB-splines. On the other hand, it verifies optimal local approximation estimates. 
\end{itemize}

\section*{Appendix 1}

Now we present the proof of Lemma~\ref{L:auxiliar characterization}. Here, if $A$ and $B$ are two knot vectors (i.e., sequences), we say that $A\subset B$ if $A$ is a subsequence of $B$. On the other hand, given two arbitrary sequences $A$ and $B$, we denote by $A\cap B$ the largest subsequence of $A$ and $B$.

\begin{proof}[Proof of Lemma~\ref{L:auxiliar characterization}]
For clarity of presentation, we consider first the univariate case $d=1$. Let $p=p_1$ be the polynomial degree and let $\ell$ be fixed satisfying $0\le \ell\le n-2$. Let $\Xi_{\ell}$ and $\Xi_{\ell+1}$ be the open knot vectors associated to the spline spaces $\VV_\ell$ and $\VV_{\ell+1}$, respectively. Let
$$\Xi_{\ell+1}:=\{\xi_1,\xi_2,\dots,\xi_{\#\BB_{\ell+1}+p+1}\}.$$
Let $\beta_{\ell+1}\in\BB_{\ell+1,\omega_{\ell+1}}$.  Then, there exists $Q\in\QQ_{\ell+1}$ such that $Q\subset\supp\beta_{\ell+1}$ and $Q\subset\omega_{\ell+1}$. Thus, there exists $k$ such that $Q=[\xi_k,\xi_{k+1}]$. By the definition of $\omega_{\ell+1}$, we have that $$\tilde Q=[\xi_{k-p},\xi_{k+p+1}]\subset\Omega_{\ell+1}.$$ 
Notice that
$$\Xi_{\tilde Q}:=\{\xi_{k-p},\dots,\xi_{k+p+1}\}$$
consists of $2p+2$ consecutive knots in $\Xi_{\ell+1}$. Since $\Xi_{\ell+1}$ is obtained from $\Xi_\ell$ by dyadic refinement, we have that $\#(\Xi_{\tilde Q}\cap\Xi_\ell)\ge p+1$. If $\#(\Xi_{\tilde Q}\cap\Xi_\ell)= p+1$, it is easy to check that we can add one knot $\hat\xi\in\Xi_\ell$ such that $\#((\Xi_{\tilde Q}\cup\{\hat\xi\})\cap\Xi_\ell)= p+2$ and $\{\xi\in(\Xi_{\tilde Q}\cup\{\hat\xi\})\cap\Xi_\ell\}\subset\Omega_{\ell+1}$. Therefore, there exists $\Xi_{\tilde Q}^\ell\subset \Xi_{\ell}$ such that
\begin{equation}\label{E:aux tricky}
 \#\Xi_{\tilde Q}^\ell\ge p+2,\qquad (\Xi_{\beta_{\ell+1}}\cap\Xi_\ell)\subset(\Xi_{\tilde Q}\cap\Xi_\ell)\subset\Xi_{\tilde Q}^\ell,\qquad \{\xi\in\Xi_{\tilde Q}^\ell\}\subset\Omega_{\ell+1}
\end{equation}
Let $r:=\#(\Xi_{\beta_{\ell+1}}\cap\Xi_\ell)$ and notice that $r\le p+1$. We consider two cases:
\begin{enumerate}
 \item[(i)] $\min\Xi_{\beta_{\ell+1}}$ or $\max\Xi_{\beta_{\ell+1}}$ matches a knot in $\Xi_{\ell}$: By~\eqref{E:aux tricky}, there exists $\beta_\ell\in\BB_\ell$ such that $\Xi_{\beta_{\ell+1}}\cap\Xi_{\ell}\subset \Xi_{\beta_\ell}\subset \Xi_{\tilde Q}^\ell$. Thus, $\beta_{\ell+1}\in\CC(\beta_{\ell})$ and $\supp\beta_\ell\subset\Omega_{\ell+1}$. 
 \item[(ii)] Neither $\min\Xi_{\beta_{\ell+1}}$ nor $\max\Xi_{\beta_{\ell+1}}$ match knots in $\Xi_{\ell}$: Since $\Xi_{\beta_{\ell+1}}=p+2$, in this case $r\le p$. Thus, there exists $\beta_\ell\in\BB_\ell$ such that $\Xi_{\beta_{\ell+1}}\cap\Xi_{\ell}\subset \Xi_{\beta_\ell}\subset \Xi_{\tilde Q}^\ell$. Again, $\beta_{\ell+1}\in\CC(\beta_{\ell})$ and $\supp\beta_\ell\subset\Omega_{\ell+1}$. 
\end{enumerate}

Finally, for the multivariate case $d>1$, we can apply this argument in each coordinate direction. 
\end{proof}



\begin{thebibliography}{1}
%

\bibitem[BuGa15]{BuGa15}
A. Buffa and E.M. Garau, \emph{A posteriori error estimators
for hierarchical B-spline discretizations}, in preparation, 2015.


\bibitem[BGGS15]{BGGS15}
A. Buffa, E.M. Garau, C. Giannelli and G. Sangalli, \emph{On quasi-interpolation in spline spaces}, in preparation, 2015.


\bibitem[BuGi15]{BuGi15}
A. Buffa and C.~Giannelli, \emph{Adaptive isogeometric methods with hierarchical splines: error 
estimator and convergence}, submitted, 2015.

\bibitem[CHB09]{CHB09}
 J.A. Cottrell, T.J.R. Hughes, and Y. Bazilevs, \emph{Isogeometric Analysis:
Toward Integration of CAD and FEA}, John Wiley \& Sons, 2009.

%
%

\bibitem[dV01]{deBoor}
C. de Boor, \emph{A practical guide to splines}, Revised edition. Applied Mathematical Sciences, 27. Springer-Verlag, New York, 2001. 

%
%
%


\bibitem[GJS12]{GJS12}
C. Giannelli, B. J\"uttler, and H. Speleers, \emph{THB-splines: the truncated basis for hierarchical splines}, Comput. Aided Geom. Design 29 (2012), no. 7, 485--498.

\bibitem[GJS14]{GJS14}
C. Giannelli, B. J\"uttler, and H. Speleers, \emph{Strongly stable bases for adaptively refined multilevel spline spaces}, Adv. Comput. Math. 40 (2014), no. 2, 459--490.

\bibitem[HCB05]{HCB05}
T.J.R. Hughes, J.A. Cottrell, and Y. Bazilevs, \emph{Isogeometric analysis: CAD, finite elements, NURBS, exact geometry and mesh refinement}, Comput. Methods Appl. Mech. Engrg. 194 (2005), no. 39-41, 4135--4195. 



\bibitem[K97]{Kraft}
R. Kraft, \emph{Adaptive and linearly independent multilevel $B$-splines}, Surface fitting and multiresolution methods (Chamonix-Mont-Blanc, 1996), 209--218, Vanderbilt Univ. Press, Nashville, TN, 1997. 


\bibitem[K98]{Kraft-thesis}
R. Kraft, \emph{Adaptive und linear unabh\"angige multilevel B-Splines und ihre Anwendungen}, Ph.D. thesis, Universit\"at Stuttgart, 1998. 


\bibitem[LLM01]{LLM01}
B.-G. Lee, T. Lyche, and K. M\o{}rken,\emph{Some examples of quasi-interpolants constructed from local spline projectors}, Mathematical methods for curves and surfaces (Oslo, 2000), 243--252, Innov. Appl. Math., Vanderbilt Univ. Press, Nashville, TN, 2001.



\bibitem[S07]{Schumaker}
L.L. Schumaker, \emph{Spline functions: basic theory}, Third edition, Cambridge Mathematical Library, Cambridge University Press, Cambridge, 2007.

%
%
%


\bibitem[SM14]{SM14}
H. Speleers and C. Manni, \emph{Effortless quasi-interpolation in hierarchical spaces}, Numer. Math., to appear, 2015.


\bibitem[VGJS11]{Vuong etal}
A.-V. Vuong, C. Giannelli, B. J\"uttler and B. Simeon, \emph{A hierarchical approach to adaptive local refinement in isogeometric analysis}, Comput. Methods Appl. Mech. Engrg. 200 (2011), no. 49-52, 3554--3567.

\end{thebibliography}
\end{document}